\documentclass[12pt,reqno]{amsart}
\usepackage[LGR,T1]{fontenc}
\usepackage[utf8]{inputenc}

\title{Minimizing relative entropy of path measures under marginal constraints}
\date{\today}
 
\author[Baradat]{Aymeric Baradat}
\address{Aymeric Baradat. Max Planck Institute for the Mathematics in the Sciences, Leipzig, Germany}
\email{aymeric.baradat@mis.mpg.de}
 
\author[Léonard]{Christian Léonard}
\address{Christian Léonard. Modal-X,  Université Paris Nanterre, France}
\email{christian.leonard@math.cnrs.fr}

 \subjclass[2010]{}

\usepackage{amssymb, amsmath, amsfonts, latexsym, enumerate}
\usepackage[usenames,dvipsnames]{pstricks}
\usepackage{epsfig}
\RequirePackage[colorlinks,linkcolor=blue,citecolor=blue,urlcolor=blue]{hyperref}
\usepackage[english]{babel}

\RequirePackage{pifont}

\usepackage{pgf, tikz}
\usepackage{pgfplots}
\usetikzlibrary{patterns}
\usetikzlibrary{plotmarks}
\usetikzlibrary{shadings}

\newtheorem{theorem}{Theorem}
\newtheorem{lemma}[theorem]{Lemma}
\newtheorem{proposition}[theorem]{Proposition}

\newtheorem{definition}[theorem]{Definition}

\newtheorem{assumptions}[theorem]{Assumptions}
\newtheorem{assumption}[theorem]{Assumption}

\theoremstyle{remark}
\newtheorem{remark}[theorem]{Remark}
\newtheorem{remarks}[theorem]{Remarks}

\numberwithin{theorem}{section}

\newcommand{\RR}{\mathbb{R}}

\newcommand{\PP}{\mathbb{P}}
\newcommand{\EE}{E}	
\newcommand{\1}{\mathbf{1}}

\newcommand\pf{_{\#}}

\renewcommand{\ae}{\textrm{-a.e.}}

	\DeclareMathOperator{\D}{d\!}

\newcommand{\Boulette}[1]{\par\medskip\noindent $\bullet$\ Proof of #1.}
\newcommand{\sbt}{\,\begin{picture}(-1,1)(-1,-3)\circle*{3}\end{picture}\ }

\newcommand{\cadlag}{c\`adl\`ag}

\newcommand\XX{\mathcal{X}}
\newcommand\XXX{\XX^2}
\newcommand{\YY}{ \mathcal{Y}}

\newcommand\PX{\mathrm{P}(\XX)}

\newcommand\PXX{\mathrm{P}(\XXX)}

\newcommand\PO{\mathrm{P}(\Omega)}
\newcommand\MO{\mathrm{M}(\Omega)}

\newcommand\OO{\Omega}

\newcommand\CO{\mathrm{C}_b(\Omega)}

\newcommand\ii{{[0,1]}}

\newcommand\IX{\int_{\XX}}

\newcommand\IO{\int_\Omega}

\newcommand{\TT}{\mathcal{T}}
\renewcommand{\SS}{\mathcal{S}}
\newcommand{\zz}{\mathcal{Z}}
\newcommand{\rr}{ \mathbf{r}}
\newcommand{\pp}{ \mathbf{p}}
\newcommand{\qq}{ \mathbf{q}}

\begin{document}
\begin{abstract}
	We study generalizations of the Schr\"odinger problem in statistical mechanics in two directions: when the density is constrained at more than two times, and when the joint law of the initial and final positions for the particles is prescribed. This is done in agreement with the so-called Br\"odinger problem recently introduced to regularize Brenier's variational model for incompressible fluids.
	
	We recover generalizations of the standard factorization result for the Radon-Nikodym derivative of the solution $P$ with respect to the reference measure $R$: this density can be written in terms of an \emph{additive functional} on the set of constrained times. 
	
	The specificity of this work is that we place ourselves in the case when $R$ is Markov (or reciprocal), and that we use Markovian methods rather than classical convex analysis arguments. In this setting, it appears that a natural assumption to be made on the reference measure $R$ is of irreducibility type.
\end{abstract}
\maketitle 
\tableofcontents

\section{Introduction}
In this paper, we are interested in the structure of the solutions to a class of entropy minimization problems among path measures: i.e.\ measures on a path space.  More specifically,  for any path measures $P \ll R$, the relative entropy of $P$ with respect to $R$ is defined by the formula:
\[
H(P|R):=\EE_P\left[\log\frac{\D P}{\D R}\right],
\]
(a  precise definition is given at Section \ref{sec:unbounded}, in particular when $R$ is unbounded) and we minimize $H(\sbt|R)$ under  various marginal constraints, when  the reference path measure $R$ is Markov, or sometimes reciprocal. The definition of a reciprocal path measure and its basic properties are recalled at Definition \ref{def:reciprocal}; in particular, any Markov measure is reciprocal.

\subsection*{Schrödinger problem}

The most classical of these problems is unquestionably the Schrödinger problem and dates back to the 30's with the original articles~\cite{sch31,sch32} by Schrödinger himself. This problem can be informally stated in the following way. Take a population of independent particles uniformly distributed in a box at the initial time, and evolving along Brownian paths. Then at any positive time, one still expects the density of particles to be approximately uniform. But suppose that a very rare event occurs, and that one measures at time  $t=0$ and $T>0$ a density of particles  far from being uniform (in all this text, we will set without loss of generality $T=1$). Conditionally on this rare event, what is the most likely statistical behavior of the particles of the system? The theory of large deviations gives the following answer. If one calls $R$ the law of the Brownian motion starting from the Lebesgue measure, $n$ the number of particles, $Y_1, \dots , Y_n$ the path of each particle and $\mu,\nu$ the observed densities at times $0$ and $1$, then the empirical measure
\[
\frac{1}{n} \sum_{i=1}^n \delta_{Y_i}
\]
must be close to the (unique if exists) minimizer $P$ of the entropy $H (\sbt  | R)$ under the constraints that the initial marginal of $P$ is $\mu$, and its final marginal is $\nu$. The Schödinger problem consists in finding $P$ as a function of $\mu$ and $\nu$. For a precise statement of this problem in a more general setting, see \eqref{eq-11}. We refer to the survey~\cite{leo13} for both a historical presentation and a precise statement of most of the known results about the Schr\"odinger problem.

There has been renewed interest in this problem since we understood its link with the theory of optimal transport. In fact, the Schrödinger problem can be seen as a regularized version of the classical quadratic optimal transport (see \cite{Mika04,leonard2012schrodinger,leo13}) and it is now used to compute numerically the solutions to transport problems~\cite{cuturi2013sinkhorn,benamou2015iterative} using the Sinkhorn algorithm~\cite{sinkhorn1964relationship,sinkhorn1967diagonal}. 

 One of the main results of this theory, which is crucial in the applications as it allows the use of the Sinkhorn algorithm, is the factorization property of $P$: there exist $f$ and $g$ two measurable functions such that if $(X_t)_{t \in [0,1]}$ denotes the canonical process on the set of continuous paths,
\begin{equation}
\label{eq:fg_transform}
P = f(X_0) g(X_1) R.
\end{equation}
This property is classically derived by means of various analytic techniques, sometimes coupled with  geometric analogies~\cite{Beu60,Csi75,follmer1988random,BLN,RT93,RT98,Leo01b}, but we prefer to understand it via a different approach: through Markovian considerations, in the spirit of \cite{leonard2014reciprocal}. More precisely, if $R$ is Markov, then $P$ needs to be Markov (see Lemma~\ref{res-05}), and the density of $P$ with respect to $R$ needs to be a function of $X_0$ and $X_1$ (see Theorem~\ref{thm:sol_schro_additive}). But under additional irreducibility assumptions on $R$ that  will be detailed at Section \ref{sec-sum}, the only way to satisfy these two properties is to be factorized.

\subsection*{Aim of this paper}

The goal of this article is to study how far these arguments can go when further constraints are added. The typical kinds of constraints we want to deal with are
\begin{enumerate}[(i)]
\item \label{item:marginal} marginal constraints, prescribing the law of $X_t$ under $P$ for all times  $t$ in a given set $\TT\subset\ii;$
\item \label{item:endpoint} endpoint constraints, prescribing the joint law of $(X_0, X_1)$ under $P$.
\end{enumerate}

Our main motivation is the understanding of the so-called Brödinger problem, which is a mixture of the Schrödinger problem just described and the Brenier problem for perfect fluids, presented and studied for the first time in~\cite{bre89}. It consists in finding the minimizer $P$ of relative entropy with respect to the reversible Brownian motion (say on the flat torus~$\mathbb{T}^d$) when the law of $X_t$ under $P$ is prescribed to be the Lebesgue measure at all times $t$ (this is  incompressibility), and an endpoint constraint is added. The reader can find information on the Brödinger problem in the recent papers~\cite{arn17,baradat2018existence,baradat2019small}.

By adding the constraints~\eqref{item:marginal} and~\eqref{item:endpoint}, the factorization property~\eqref{eq:fg_transform} translates formally into
\begin{equation}
\label{eq:general_factorization}
\frac{\D P}{\D R} = \exp\Big( \eta(X_0, X_1) + A([0,1])\Big),
\end{equation} 
where $\eta$ is measurable and comes from the endpoint constraint, and $A$ is what we call an \emph{additive functional} only charging $\TT$ and comes from the density constraints. We will be more specific in Definition~\ref{def:additive_functional} below, but $A$ is essentially a random finitely additive measure (or \emph{content}) which only charges $\TT$ and such that for all interval $I$, $A(I)$ only depends on the values of the canonical process $X_t$ for $t \in I$. An important observation is that in the case of endpoint constraints (leading to $\eta \neq 0$), $P$ is not Markov, even when $R$ is Markov. This leads to additional difficulties. However, it inherits a more general property, namely the reciprocity, which is still tractable in some cases.
	
We could not reach formula~\eqref{eq:general_factorization} in full generality; see Remark~\ref{rem:general_doesnt_work}. However the main results of this paper are two particular cases. In Theorem~\ref{thm:sol_schro_additive}, we treat the case when there is no endpoint constraint (and hence $\eta=0$). In Theorem~\ref{thm:bro}, we solve the case when there is an endpoint constraint, but $\TT$ is finite. 

As a consequence, we do not get a fully satisfactory description of the solutions of the Brödinger problem, but rather of a discrete version of it where only finitely many time marginals are prescribed.

\subsection*{Notation} In the whole paper, $\XX$ denotes a Polish space, and the \emph{path space} $\Omega$ is the set of all continuous curves from $[0,1]$ to $\XX.$  For all $t \in [0,1]$, $X_t$ (called the canonical process at time $t$) is the evaluation map at time $t$, that is for all $\omega \in \Omega$, $X_t(\omega) = \omega(t)$.  If $\mathcal{Z}$ is a measurable space, $\mathrm{M}(\mathcal{Z})$ and $\mathrm{P}(\mathcal{Z})$ will be the sets of measures and probability measures on $\mathcal{Z}$ respectively. We recall that when endowed with the  topology of uniform convergence, $\Omega$ is a Polish space and its Borel $ \sigma$-field is the natural $ \sigma$-field $ \sigma(X_t; 0\le t\le 1)$ generated by the canonical process $(X_t) _{ 0\le t\le 1}$. In that case, $\MO$ and $\PO$ stand for the set of Borel measures on $\Omega$ (that we call \emph{path measures}) and the set of Borel probability measures on $\Omega$. 

If $\mathcal{T}$ is a subset of $[0,1]$, we denote $X_{\mathcal{T}}:=(X_t) _{ t\in \mathcal{T}}.$   If $R \in \MO$ is conditionable in the sense given in Definition \ref{def:conditionable} and $t_1, \dots, t_p \in [0,1]$, $R_{t_1, \dots, t_p}$  stands for the law of $(X_{t_1}, \dots, X_{t_p})$ under $R$ and $R_{\mathcal{T}}$  stands for the law of $X_{\mathcal{T}}$ under $R$. \\ More generally, if $\mathcal{Z}_1$ and $\mathcal{Z}_2$ are two measurable spaces, if $\rr \in \mathrm{M}(\mathcal{Z}_1)$ and if $X: \mathcal{Z}_1 \to \mathcal{Z}_2$ is measurable, then the law of $A$ under $\rr$ is denoted by~$X \pf \rr$. 
If $\mathcal{Z}_1$ is Polish and $X \pf \rr$ is $\sigma$-finite, according to the disintegration theorem, then for $X \pf \rr$-almost all $x$, the conditional probability $\rr( \sbt  | X = x)$ exists. It will be called $\rr^x$, and $\rr^X : \mathcal{Z}_1 \to \mathrm{P}(\mathcal{Z}_1)$ will be the random variable with measure values almost everywhere defined by
\[
\rr\mbox{-a.e.} \quad X = x \quad  \Rightarrow \quad \rr^X = \rr^x.
\]

\subsection*{Outline of the paper}
Let $R\in\MO$ be a fixed reference path measure and $P$ be any  path probability measure $P\in\PO$  dominated by $R$, i.e.\ $P$ is absolutely continuous with respect to $R$: $P\ll R$. 

At Section \ref{sec-mark}, we derive characterizations in terms of additive functionals of the Radon-Nikodym density $\D P/ \D R$ in the case when both $P$ and $R$ are Markov.
This result is stated at Theorem~\ref{res-02}: $\D P / \D R = \exp(A([0,1]))$ for some additive functional $A$.

In Section~\ref{sec-sum}, we ask the following question. Let us assume that for some times $0 \leq s < u < t \leq 1$ and some measurable functions $a$, $b$ and $c$, we have:
	\begin{equation*}
	f(X_s,X_t)= a(X _{ [s,u ]}) + b( X _{ [u ,t]}), \qquad R\ae
	\end{equation*}
In which cases is it true that there are two measurable functions $f_s$ and $f_t$ such that:
\begin{equation*}
f(X_s,X_t) = f_s(X_s) + f_t(X_t), \qquad R\ae?
\end{equation*}
This is where the irreducibility of $R$ will come into play, as we will prove in Lemma~\ref{lem:sum} that a sufficient condition is to ask $R$ to be reciprocal (see Definition~\ref{def:reciprocal}) and irreducible (see Assumption~\ref{ass:irreducibility}). We also provide counterexamples when one of these two conditions fails to be satisfied.

Then, we explore  at Section~\ref{sec-min} some consequences for multimarginal entropic minimization problems. In the Schr\"odinger case, when $R$ is Markov, then $P$ is Markov. If in addition, $R$ is irreducible, the result of Section~\ref{sec-sum} let us build from the additive functional $A$ given by Theorem~\ref{res-02} another additive functional that only charges $\TT$ and which coincide with $A$ on the full interval $[0,1]$. This is stated at Theorem~\ref{thm:sol_schro_additive}.

In Section~\ref{sec:finitely-many}, we apply Theorem~\ref{thm:sol_schro_additive} to get a description of the solution of the Schr\"odinger problem when there is a finite number of density constraints (Theorem~\ref{thm:discrete_schro}). We also get~\eqref{eq:general_factorization} in the Br\"odinger case (\emph{i.e.} when adding an endpoint constraint), still when there is a finite number of density constraints (Theorem~\ref{thm:bro}). In this case, we can even assume that $R$ is only reciprocal. Theorem~\ref{thm:bro} is a consequence of Theorem~\ref{thm:discrete_schro} once noticing that any reciprocal process can be transformed into a Markov process by enlarging the space of states and by "folding" the trajectories; see Lemma~\ref{lem:reciprocal_markov}. As far as we know, this argument is new.

 In Appendix~\ref{regularitymartingale}, we show how a regularity result by Bakry~\cite{bakry1979regularite} for two indices martingales let us get for free a regularity property for the quantity $A([s,t])$ with respect to $s$ and $t$ where $A$ is the additive functional given by Theorem~\ref{res-02}. This regularity is very mild (it is nothing but a generalization of the \emph{càdlàg} property for classical martingales), but it is sufficient to characterize $A$ by its value on a countable amount of intervals. This is often useful to prove that a property which is true $\forall t,\ \mbox{a.e.}$ is also true $\mbox{a.e.}, \ \forall t$.
 
 Finally in Appendix~\ref{app:CT}, we state and prove a standard result in measure theory that shows the relationship between disintegration and absolute continuity of measures. These results are used several times  along the proofs.

\section{Dominated Markov measures}\label{sec-mark}

\subsection*{Basic definitions}

Let us begin with the symmetric definition of the Markov property. First of all, we need to make precise what a conditionable path measure is.

\begin{definition}[Conditionable path measure]
	\label{def:conditionable}
The path measure $Q\in\MO$ is said to  be \emph{conditionable} if for all $t\in\ii,$  $Q_t$ is a $\sigma$-finite measure on $\XX$.
\end{definition}

It is shown in \cite{leonard2014some} that for any  conditionable path measure $Q\in\MO,$ the conditional expectation $\EE_Q[\sbt| X_\mathcal{T}]$ is well-defined for any $\mathcal{T}\subset\ii.$ This is the reason for this definition.

\begin{definition}[Markov measure] 
	\label{def:markov}
A path measure  $Q$ on  $\Omega$ is said to be  \emph{Markov}  if it is conditionable and 
 if for any $t\in[0,1]$ and for any events $A\in \sigma(X_{[0,t]}), B\in \sigma(X_{[t,1]})$
    \begin{equation}\label{eq-01}
    Q(A\cap B| X_t)=Q(A| X_t)Q(B| X_t), \quad Q\ae  
    \end{equation}
\end{definition}
This means that, knowing the present state  $X_t$, the future and  past   informations  $ \sigma(X_{[t,1]})$  and  $ \sigma(X_{[0,t]})$, are  $Q$-independent.  

We will represent the Radon-Nikodym derivatives between laws of Markov processes with the help of \emph{contents} (see \cite{halmos1950measure}). Contents can be understood as "finitely additive measures". 
\begin{definition}[Content]
Denote by $\mathcal{I}$ the subset of $2^{[0,1]}$ composed by every finite unions of intervals of $[0,1]$. The set $\mathcal{I}$ is clearly closed under finite unions and intersections and under complements of individual elements (such a set is sometimes called a field of sets). We say that a function $\mu$ from $\mathcal{I}$ to $[- \infty, + \infty)$ is a \emph{content} and we write $\mu \in \mathfrak{C}([0,1])$ if:
\begin{itemize}
\item $\mu(\emptyset) = 0$,
\item for all $I_1, \, I_2 \in \mathcal{I}$, $\mu(I_1 \cup I_2) = \mu(I_1) + \mu(I_2)$ whenever $I_1 \cap I_2 = \emptyset$. 
\end{itemize}
\end{definition}
\begin{remark}
\label{contentonclosedintervals}
It is a simple exercise to show that if $\nu$ is a function from the set of all intervals of the form $[s,t]$ with $s \leq t$ in $[0,1]$ such that for all $s \leq u \leq v \leq t$ in $[0,1]$,
\begin{equation*}
\begin{aligned}
&\mbox{if }\nu([u,v]) = - \infty, &&\mbox{then } \nu([s,v]) = \nu([u,t]) = \nu([s,t]) = - \infty,\\
&\mbox{else,}&&\nu([s,t]) = \nu([s,v]) + \nu([u,t]) - \nu([u,v]),
\end{aligned}
\end{equation*}
then there is a unique content $\mu$ extending $\nu$. For example, for all $s < t$ in $[0,1]$,
\[
\mu((s,t)) = \nu([s,t]) - \nu(\{s\}) - \nu(\{t\}).
\]
\end{remark}

A typical structure of the contents we will get is the following inner or outer regularity property:
\begin{definition}[Regular content]
	\label{def:regular_content}
We say that a content $\mu \in \mathfrak{C}([0,1])$ is \emph{regular} and we write $\mu \in \mathfrak{C}_r([0,1])$ if one of the two equivalent following properties is satisfied:
\begin{itemize}
\item $ \left\{\begin{array}{ll}
\forall \, 0 \leq s \leq t \leq 1, \, (s,t) \neq (0,1), \quad &\lim_{\substack{(\sigma, \tau) \to (s,t) \\
\sigma \leq s, \, \tau \geq t}}\mu([\sigma, \tau]) = \mu([s,t]),\\
\forall \, 0 \leq s < t \leq 1, & \lim_{\substack{(\sigma, \tau) \to (s,t) \\
\sigma > s, \, \tau < t}} \mu([\sigma, \tau]) \mbox{ exists},
\end{array}\right.$ 
\item[] \null
\item $ \left\{\begin{array}{ll}
\forall \, 0 \leq s < t \leq 1,  \quad& \lim_{\substack{(\sigma, \tau) \to (s,t) \\
\sigma \geq s, \, \tau \leq t}}\mu\big( (\sigma, \tau) \big) = \mu\big((s,t)\big),\\
\forall \, 0 \leq s \leq t \leq 1, \, (s,t) \neq (0,1), \quad &\lim_{\substack{(\sigma, \tau) \to (s,t) \\
\sigma < s, \, \tau > t}} \mu\big( (\sigma, \tau)\big) \mbox{ exists},
\end{array}\right.$
\end{itemize}
\end{definition}
\begin{definition}[Additive functional]
	\label{def:additive_functional}
A mapping $A :\OO\to \mathfrak{C}([0,1])$ is said to be an \emph{additive functional} if it is measurable in the sense that for every $I \in \mathcal{I}$, the function $A(I)$ is $\sigma(X_t, \, t \in I)$-measurable.
\\
We say in addition that $A$ is \emph{regular} if its values are regular.
\end{definition}
 
The proof of Theorem \ref{res-02}, the main result of the present section, relies on the following preliminary Lemma \ref{res-01}.

\subsection*{A general lemma} Lemma~\ref{res-01} gives the structure of the Radon-Nikodym derivative with respect to $\rr$ of a measure $\pp \ll \rr$ sharing some independence properties with $\rr$. The context if the following:

\begin{assumptions}[Framework for Lemma~\ref{res-01}]
	\label{ass:framework}
 Consider $\mathcal{Z}$, $\mathcal{Z}_A$, $\mathcal{Z}_B$ and $\mathcal{Z}_C$ four measurable spaces and three measurable mappings $A:\mathcal{Z} \to \mathcal{Z}_A,$ $B:\mathcal{Z}\to\mathcal{Z}_B$ and $C:\mathcal{Z}\to \mathcal{Z}_C$. We suppose that the $\sigma$-algebra on $\mathcal{Z}$ is generated by $A$, $B$ and $C$. Take $\rr \in \mathrm{M}(\mathcal{Z})$ satisfying:
 \begin{itemize}
 \item the push-forward $C \pf \rr$ is $\sigma$-finite (in particular, $\rr$ is $\sigma$-finite),
 \item under $\rr$, the mappings $A$ and $B$ are independent conditionally on $C$, \textit{i.e.} for all measurable and nonnegative functions $u$ and $v$, 
 \[
 \EE_{\rr}[u(A) v (B) | C] = \EE_{\rr}[u(A) | C] \: \EE_{\rr}[v(B)|C], \quad \rr\mbox{-a.e}.
 \]
 \end{itemize}
\end{assumptions}
 \begin{remark}
 \label{dependanceonC} In this setting, it can be shown that for all measurable and nonnegative functions $U: \mathcal{Z}_A \times \mathcal{Z}_C \to \RR_+$ and $V: \mathcal{Z}_B \times \mathcal{Z}_C \to \RR_+$, we have $\rr$-almost everywhere
 \[
 \EE_{\rr}[U(A,C) V(B,C)|C] =  \EE_{\rr}[ U(A,C)|C] \: \EE_{\rr}[V(B,C)|C].
 \]
 \end{remark}
 
The following Lemma \ref{res-01} will be used during the proof of Theorem~\ref{res-02}. It is a straightforward extension of \cite[Thm.\,1.5]{leonard2014reciprocal}.

\begin{lemma}\label{res-01}
 In the framework of Assumption~\ref{ass:framework}, consider $\pp \in \mathrm{P}(\mathcal{Z})$ such that $\pp \ll \rr$.
 \begin{itemize}
\item If under $\pp$, the measurable mappings $A: \mathcal{Z}\to \mathcal{Z}_A$ and $B: \mathcal{Z}\to \mathcal{Z}_B$ are independent conditionally on $C: \mathcal{Z}\to \mathcal{Z}_C$, then there are three nonnegative measurable functions $\alpha: \mathcal{Z}_A \times \mathcal{Z}_C \to \RR_+$, $\beta: \mathcal{Z}_B \times \mathcal{Z}_C \to \RR_+$ and $\gamma: \mathcal{Z}_C \to \RR_+$ such that $\rr$-almost everywhere,
 \begin{gather}
 \label{eq-01P} \frac{\D \pp}{\D \rr} = \alpha(A, C) \beta (B,C) \gamma(C), \\
\label{defalphabeta} \rr \ae, \quad \gamma(C) \neq 0 \Rightarrow \EE_{\rr}[\alpha(A,C)|C] = \EE_{\rr}[\beta(B,C)|C] = 1.
 \end{gather}
 In that case,
 \begin{equation}
 \label{eq:expression_abc}
 \left\{
  \begin{gathered}
 \gamma = \frac{ \D C\pf \pp }{ \D C\pf \rr },\\
  \alpha \times \gamma = \frac{ \mathrm{d} (A,C)\pf \pp }{ \mathrm{d} (A,C)\pf \rr },\\
 \beta \times \gamma = \frac{ \mathrm{d} (B,C)\pf \pp }{ \mathrm{d} (B,C)\pf \rr }.
 \end{gathered}
 \right.
 \end{equation}
 \item Conversely, if there are two nonnegative function $\alpha$ and $\beta$ such that
 \begin{equation}
 \label{eq:RN_ab}
 \frac{\D \pp}{\D \rr} = \alpha(A,C) \beta(B,C),
 \end{equation}
 then under $\pp$, $A$ and $B$ are independent conditionally on $C$.
 \end{itemize}
\end{lemma}

\begin{proof}
For the first point, let
 \[
 D := \frac{\D \pp}{\D \rr}.
 \]
Because $C \pf \rr$ is $\sigma$-finite, so are $(A,C) \pf \rr$ and $(B,C) \pf \rr$. As a consequence, the following functions are well defined $\rr$-almost everywhere:
\begin{gather*}
\gamma (C) := \EE_\rr[D|C], \\
\alpha (A,C) := \left\{ \begin{array}{l l} \displaystyle{\frac{\EE_\rr[D|A,C]}{\gamma(C)}} &\mbox{if } \gamma(C) \neq 0, \\
0 & \mbox{else}, \end{array} \right. \\
\beta (B,C) := \left\{ \begin{array}{l l} \displaystyle{\frac{\EE_\rr[D|B,C]}{\gamma(C)}} &\mbox{if } \gamma(C) \neq 0, \\
0 & \mbox{else}. \end{array}\right. 
\end{gather*}
On $\{\gamma(C) \neq 0 \}$, it is easily shown that $\rr$-almost everywhere,
\[
 \EE_{\rr}[\alpha(A,C)|C] = \EE_{\rr}[\beta(B,C)|C] = 1.
\]

Moreover $\rr$-almost everywhere on $\{ \gamma(C) = 0 \}$, $D = 0$, so the identities
\begin{gather*}
\EE_{\rr}[D|A,C] = \alpha(A,C) \gamma(C), \\
\EE_{\rr}[D|B,C] = \beta(B,C) \gamma(C)
\end{gather*}
hold $\rr$-almost everywhere.

It remains to show
\begin{equation}
\label{eq:D=abc}
D = \alpha (A,C) \beta (B,C) \gamma(C).
\end{equation}

Let $u$ be a nonnegative measurable function. Let us compute $\EE_{\pp}[u(A)|C]$. For all measurable and nonnegative function $w$, we have on the one hand:
\begin{align*}
E_{\pp}[u(A) w(C)] &= \EE_{\pp}[\EE_{\pp}[u(A)|C] w(C)] \\
&= \EE_{\rr}[D \EE_{\pp}[u(A)|C] w(C)]\\
&= \EE_{\rr}[\EE_{\rr}[D|C] \: \EE_{\pp}[u(A)|C] w(C)]\\
&= \EE_{\rr}[\gamma(C) \EE_{\pp}[u(A)|C] w(C)],
\end{align*}
and on the other hand:
\begin{align*}
\EE_{\pp}[u(A) w(C)] &= \EE_{\rr}[Du(A) w(C)] \\
&= \EE_{\rr}[\EE_{\rr}[Du(A)|C] w(C)]\\
&= \EE_{\rr}[\EE_{\rr}[\EE_{\rr}[D|A,C] u(A)|C] w(C)]\\
&= \EE_{\rr}[\alpha(A,C) u(A)|C] \gamma(C) w(C)]
\end{align*}
So we have $\rr$-almost everywhere:
\[
\gamma(C) \EE_{\pp}[u(A)|C] = \gamma(C) \EE_{\rr}[\alpha(A,C) u(A)|C],
\]
and $\pp$-almost everywhere:
\[
\EE_{\pp}[u(A)|C] = \EE_{\rr}[\alpha(A,C) u(A)|C].
\]
In the same way, for all measurable and nonnegative function $v$, $\pp$-almost everywhere:
\[
\EE_{\pp}[v(B)|C] = \EE_{\rr}[\beta(B,C) v(B)|C].
\]
Using the fact that the $\sigma$-algebra on $\mathcal{Z}$ is $\sigma(A,B,C)$, to get~\eqref{eq:D=abc}, it suffices to show that for all $u$, $v$ and $w$ measurable and nonnegative:
\[
\EE_{\pp}[u(A)v(B)w(C)] = \EE_{\rr} [\alpha(A,C) \beta(B,C) \gamma(C) u(A) v(B) w(C)].
\]
So let us take $u$, $v$ and $w$ such functions. Because of the independence assumption,
\begin{align*}
\EE_{\pp}[u(A)v(B)w(C)] &= \EE_{\pp}[ \EE_{\pp}[u(A) v(B) | C]w(C)] \\
&= \EE_{\pp}[ \EE_{\pp}[u(A)|C] \: \EE_{\pp}[ v(B) | C] w(C)].
\end{align*}
And using the formulas computed just before,
\[
\EE_{\pp}[u(A)v(B)w(C)] = \EE_{\pp}[ \EE_{\rr}[\alpha(A,C) u(A)|C] \: \EE_{\rr}[\beta(B,C) v(B)|C] w(C)].
\]
Now, because of Remark \ref{dependanceonC},
\begin{align*}
\EE_{\pp}[u(A)v(B)w(C)] &= \EE_{\pp}[ \EE_{\rr}[\alpha(A,C) u(A) \beta(B,C) v(B)|C] w(C)] \\
&=\EE_{\rr}[D\EE_{\rr}[\alpha(A,C) u(A) \beta(B,C) v(B)|C] w(C)] \\
&=\EE_{\rr}[\gamma(C)\EE_{\rr}[\alpha(A,C) u(A) \beta(B,C) v(B)|C] w(C)]\\
&=\EE_{\rr} [\alpha(A,C) \beta(B,C) \gamma(C) u(A) v(B) w(C)],
\end{align*}
and the result follows.

The fact that \eqref{eq-01P} and \eqref{defalphabeta} imply \eqref{eq:expression_abc} follows from conditioning \eqref{eq-01P} with respect to $C$, $(A,C)$ and $(B,C)$ respectively, and from using the independence property of $\rr$.

For the second point, we suppose that there are $\alpha$ and $\beta$ such that~\eqref{eq:RN_ab} holds, and we want to show that for all $u$ and $v$ measurable and nonnegative, we have
\begin{equation*}
 \EE_{\pp}[u(A) v (B) | C] = \EE_{\pp}[u(A) | C] \: \EE_{\pp}[v(B)|C], \quad \pp\mbox{-a.e}.
\end{equation*}

This follows from the following, obtained using the same kind of computations as before:
 \begin{gather*}
 \pp\ae,\qquad \EE_{\rr}[\alpha(A,C)|C]>0 \quad \mbox{and}\quad \EE_{\rr}[\beta(B,C)|C],\\
 \EE_{\pp}[u(A)v(B)|C] = \frac{\EE_{\rr}[u(A) \alpha(A,C)|C]\EE_{\rr}[v(B) \beta(B,C)|C]}{\EE_{\rr}[\alpha(A,C)|C]\EE_{\rr}[\beta(B,C)|C]}.  
 \end{gather*}
 It remains to apply this formula first for $v=1$, then for $u=1$, and finally for general $u$ and $v$.
\end{proof}

\subsection*{Density of a Markov process}
We are now ready to give the main result on the form of the Radon-Nikodym derivative between two laws of Markov processes.

\begin{theorem}\label{res-02}
Let $R$ be a reference Markov measure  and let  $P\ll R$  a probability measure  dominated by $R$ with finite entropy. The three following assertions are equivalent.
\begin{enumerate}
\item \label{markov} The measure $P$ is Markov;
\item \label{additivefunctionnal} There is an additive functional $A$ such that
\[
\frac{\D P}{\D R} = \exp\Big(A\big([0,1]\big)\Big);
\]
\item \label{radditivefunctionnal} There is a regular additive functional $A$ such that
\[
\frac{\D P}{\D R} = \exp\Big(A\big([0,1]\big)\Big).
\]
\end{enumerate}
\end{theorem}

\begin{proof}
We show \eqref{radditivefunctionnal} $\Rightarrow$ \eqref{additivefunctionnal} $\Rightarrow$ \eqref{markov} $\Rightarrow$ \eqref{radditivefunctionnal}.

\noindent$\bullet$\ 
The implication \eqref{radditivefunctionnal} $\Rightarrow$ \eqref{additivefunctionnal} is obvious.

\noindent$\bullet$\ 
For \eqref{additivefunctionnal} $\Rightarrow$ \eqref{markov}, take $t \in [0,1]$. Then
\[
\frac{\D P}{\D R} = \exp\Big( A\big( [0,t] \big) \Big) \exp\Big( A\big( (t, 1] \big) \Big).
\]
Because of the measurability property of an additive functional, $ \exp\big( A\big( [0,t] \big) \big)$ is a function of $X_{[0,t]}$ and $\exp\big( A\big( (t, 1] \big) \big)$ is a function of $X_{[t,1]}$. We can use the second point of Lemma~\ref{res-01} with $A = X_{[0,t]}$, $B = X_{[t,1]}$ and $C = X_t$ to deduce that under $P$, $X_{[0,t]}$ and $X_{[t,1]}$ are independent conditionally on $X_t$. Because this is true for every $t$, $P$ is Markov.

\noindent$\bullet$\ 
To prove \eqref{markov} $\Rightarrow$ \eqref{radditivefunctionnal}, set for every $s \leq t$ in $[0,1]$
\[
D_{s,t} := \EE_R \bigg[ \frac{\D P}{\D R} \bigg| X_{[s,t]} \bigg],
\]
which are obviously $\sigma(X_{[s,t]})$-measurable. In Appendix~\ref{regularitymartingale}, we show how a result from~\cite{bakry1979regularite} can be used to show that up to a modification, 
\begin{gather*}
\forall \, 0 \leq s \leq t \leq 1, \, (s,t) \neq (0,1), \quad \lim_{\substack{(\sigma, \tau) \to (s,t) \\
\sigma \leq s, \, \tau \geq t}} D_{\sigma, \tau} = D_{s,t},\\
\forall \, 0 \leq s < t \leq 1, \quad \lim_{\substack{(\sigma, \tau) \to (s,t) \\
\sigma > s, \, \tau < t}} D_{\sigma, \tau} \mbox{ exists},
\end{gather*}
see Figure~\ref{regularityD}.
\\
We now define for all $0 \leq s \leq t \leq 1$
\[
A([s,t]) := \log D_{s,t} \in [- \infty, + \infty).
\]
It is well defined for all $\omega$ and all $0 \leq s \leq t \leq 1$, and it has the two regularity properties
\begin{gather*}
\forall \, 0 \leq s \leq t \leq 1, \, (s,t) \neq (0,1), \quad \lim_{\substack{(\sigma, \tau) \to (s,t) \\
\sigma \leq s, \, \tau \geq t}}A([\sigma, \tau]) = A([s,t]),\\
\forall \, 0 \leq s < t \leq 1, \quad \lim_{\substack{(\sigma, \tau) \to (s,t) \\
\sigma > s, \, \tau < t}} A([\sigma, \tau]) \mbox{ exists}.
\end{gather*}
 Let us follow Remark~\ref{contentonclosedintervals}. Let $s \leq u \leq v \leq t$ in $[0,1]$ and apply the first point of Lemma~\ref{res-01} with $\rr = R_{[s,t]}$, $\pp = P_{[s,t]}$ (these measures are still Markov), $A = X_{[s,v]}$, $B = X_{[u,t]}$ and $C = X_{[u,v]}$. We end up with three nonnegative measurable functions $\alpha$, $\beta$ and $\gamma$ such that $R$-almost everywhere,
\begin{gather*}
\exp\Big( A \big( [s,t] \big) \Big) =  \frac{\D P_{[s,t]}}{\D R_{[s,t]}} = \alpha(X_{[s,v]}) \times \beta(X_{[u,t]}) \times \gamma(X_{[u,v]}),\\
\exp\Big( A \big( [s,v] \big) \Big) =  \frac{\D P_{[s,v]}}{\D R_{[s,v]}} = \alpha(X_{[s,v]}) \times \gamma(X_{[u,v]}),\\
\exp\Big( A \big( [u,t] \big) \Big) =  \frac{\D P_{[u,t]}}{\D R_{[u,t]}} = \beta(X_{[u,t]}) \times \gamma(X_{[u,v]}),\\
\exp\Big( A \big( [u,v] \big) \Big) =\frac{\D P_{[u,v]}}{\D R_{[u,v]}} = \gamma(X_{[u,v]}).
\end{gather*}
We easily conclude that $R$-almost everywhere:
\[
 A \big( [u,v] \big) = - \infty \quad \Rightarrow \quad  A \big( [s,t] \big) =  A \big( [s,v] \big) =  A \big( [u,t] \big) = -\infty,
\]
and that if it is not the case:
\[
 A \big( [s,t] \big) =  A \big( [s,v] \big) +  A \big( [u,t] \big) -  A \big( [u,v] \big).
\]
The measurability property for the intervals that are not of the form $[s,t]$ comes from the fact that $\sigma(X_{[s,t]}) = \sigma(X_{(s,t)})$, so that for example, $A\big( (s,t) \big) = A([s,t]) - A(\{ s \}) - A(\{ t \})$ is $\sigma(X_{[s,t]})$-measurable and so $\sigma(X_{(s,t)})$-measurable. Finally, the measurability property for all sets of $\mathcal{I}$ is obtained by the additivity property of $A$.
\\
The only remaining  thing  to show is that this additivity property is valid $R$-almost everywhere for all $0 \leq s \leq u \leq v \leq t \leq 1$, and not only for all $0 \leq s \leq u \leq v \leq t \leq 1$, $R$-almost everywhere. But it is true $R$-almost everywhere for all $s \leq u \leq v \leq t$ in a countable dense subset of $[0,1]$ and it is easy to pass to the limit thanks to the regularity property of $A$. We set $A \equiv 0$ on the $\omega$'s for which the additivity does not hold for all $0 \leq s \leq u \leq v \leq t \leq 1 $, so that the property is satisfied for all $\omega$.
\end{proof}
\begin{remark} 
This decomposition is far from being unique. To illustrate this, take $( \varphi_t)_{t \in [0,1]}$ and $( \psi_t)_{t \in [0,1]}$ two families of measurable functions, such that $ \varphi_1 = \psi_0 = 0$ and such that $R$-almost everywhere, $t \mapsto \varphi_t(X_t)$ is right continuous, left limited and $s \mapsto  \psi_s(X_s)$ is left continuous, right limited. Then, if $A$ is an additive functional one can define $B$ on the closed intervals of $[0,1]$ by:
\[
B([s,t]) = A([s,t]) + \varphi_t(X_t) + \psi_s(X_s).
\]
It is easy to check with the help of Remark~\ref{contentonclosedintervals} that $B$ can be extended in a regular additive functional, with:
\[
A([0,1]) = B([0,1]).
\]
\end{remark}

\section{Writing a function as a sum}
\label{sec-sum}

In this section, we give a framework in which a function of two variables $f = f(x,y)$ can be decomposed as a sum $f(x,y) = \alpha(x) + \beta(y)$. We also provide counter-examples when only a part of the assumptions is satisfied.

The general framework is the following. Suppose that for a path measure $R \in \MO$, the following holds:

\begin{equation}
\label{eq-19}
f(X_s,X_t)= a(X _{ [s,u ]}) + b( X _{ [u ,t]}),
\quad R\ae
\end{equation}
where $0 \leq s < u < t \leq 1$ and $f$, $a$ and $b$ are $[ - \infty, \infty)$-valued measurable functions. This means that a function only depending on the position of the canonical process at times $s$ and $t$ can be decomposed as a sum of a function only depending on the \emph{beginning} of the trajectory and a function only depending on the \emph{end} of the trajectory. The question is to know if it is possible to find two measurable functions  $f_s$ and $f_t$ such that:
  \begin{equation}
\label{eq-08}
f(X_s,X_t)=f_s(X_s)+f_t(X_t),
\quad R\ae
\end{equation}
We will provide counter-examples to show that it is not always the case: it is necessary to make some assumptions on $R$ to be able to conclude. As it will be revealed by the counter-examples,  two types of assumptions are needed.
\begin{enumerate}[1.]
	\item \label{ass:formal_irreducibility} The first assumption is of irreducibility type: there is a non-negligible set of positions $y \in \XX$ such that:
	\begin{equation}
	\label{eq:evbd_goes_to_y}
	R((X_s,X_t)\in\sbt) \ll  R((X_s,X_t)\in\sbt|X_u=y),
	\end{equation}
where $0\le s\le u\le t\le 1.$
	Roughly speaking, for such $y$, if some trajectories of the process go from $X_s = x$ to $X_t = z$, then there are also trajectories such that $X_s = x$ and $X_t = z$ with the additional property that $X_u = y$.
	\item \label{ass:formal_independence} The second assumption concerns the independence properties of $R$ at time $u$: for the positions $y$ satisfying the first assumption it is possible to find for each $x \in \XX$ a trajectory $\alpha^x$ on the set of times $[s,u]$ joining $x$ to $y$, and for each $z \in \XX$ a trajectory $\beta^z$ on the set of times $[u,t]$ joining $y$ to $z$, such that:\footnote{Here and in the following, if $\rr$ is a measure on a Polish space and $X, Y$ are random variables defined on this Polish space, we denote by $\rr(X \in \bullet | Y=y)$ the law of $X$ knowing $Y=y$ under $\rr$, when it is well defined.}
	\begin{align*}
	R(X _{ [s,u]}\in\sbt| X_u= y)
	&\ll R(X _{ [s,u]}\in\sbt| X _{ [u,t]}= \beta^{X_t})\quad R\ae \\
	R(X _{ [u,t]}\in\sbt| X_u= y)
	&\ll R(X _{ [u,t]}\in\sbt| X _{ [s,u]}= \alpha^{X_s})\quad R\ae
	\end{align*}
	In other terms, any \emph{beginning} of trajectory followed by the process can be extended by the ends of trajectory $\beta$, and any \emph{end} of trajectories followed by the process can be extended by the beginnings of trajectories $\alpha$
\end{enumerate}
Under these two assumptions, it suffices to choose:
\begin{equation*}
f_s(X_s) := a(\alpha^{X_s}) \qquad \mbox{and} \qquad f_t(X_t) := b(\beta^{X_t}),
\end{equation*}
in \eqref{eq-08}.
However, the difficulty is in general to find situations where the trajectories $\alpha^x$ and $\beta^z$ can be built as measurable functions of $x$ and $z$ respectively, and to deal with the negligible sets. It is possible to state very general assumptions under which a measurable selection theorem~\cite{wagner1977survey} allows us to achieve these goals (this implies  using the axiom of choice). We shall rather follow another path by assuming that $R$ is reciprocal (instead of item~\ref{ass:formal_independence}.), and that it satisfies an additional irreducibility property (see Assumption~\ref{ass:irreducibility} below, instead of item~\ref{ass:formal_irreducibility}.)  Under these  requirements,  assumptions \ref{ass:formal_irreducibility}.\ and \ref{ass:formal_independence}.\  hold for any $0\leq s < u < t \leq 1$ and for any $y \in \XX$, and the axiom of choice is not  necessary.

\subsection*{Counterexamples}
 Let us provide two counter-examples when assumptions~\ref{ass:formal_irreducibility}.\ and~\ref{ass:formal_independence}.\ above are not fulfilled.
\begin{enumerate}[(a)]
	\item
	\emph{A path measure $R$ not satisfying the assumption \ref{ass:formal_independence}}.
	\begin{figure}
		\includegraphics{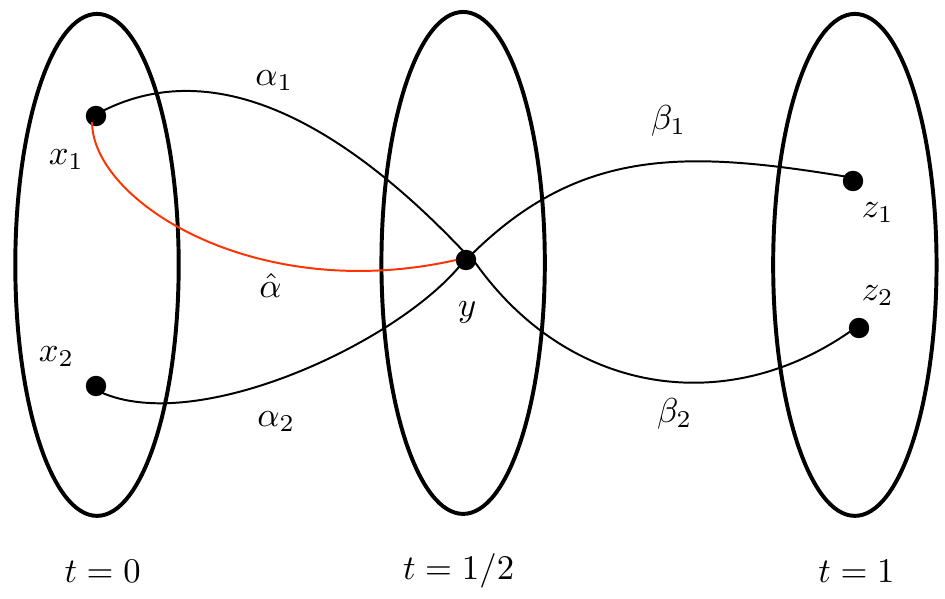}
		\caption{A counter-example when assumption~\ref{ass:formal_independence}. is dropped}\label{fig-02}
	\end{figure}
	We consider the simple setting depicted at Figure \ref{fig-02}, where only four paths are allowed: the path measure  $R$ is  uniform  on $ \left\{ \alpha_1 \beta_2, \hat \alpha \beta_1, \alpha_2 \beta_1, \alpha_2 \beta_2\right\} $ with obvious notation, and it is assumed that $x_1\not=x_2$, $z_1\not= z_2$ and $ \hat \alpha\not = \alpha_1$.
	Clearly the assumption~\ref{ass:formal_independence} is not satisfied. In particular $R$ fails to be Markov or reciprocal. We exhibit a function $f$ satisfying \eqref{eq-19} but not \eqref{eq-08}.
	\\
	The function $f(X_0,X_1)$ is specified by:
	$$f(x_1,z_1)=1\quad \textrm{and} \quad f(x_1,z_2)=f(x_2,z_1)=f(x_2,z_2)=0,$$ while the functions $a$ and $b$ are given by:  
	$$a( \hat\alpha)=1,\quad a( \alpha_1)= a( \alpha_2)=0, \quad \textrm{and} \quad b( \beta_1)=b( \beta_2)=0.$$
	Obviously, we have: $f(X_0,X_1)=a(X _{ [0,1/2]}) +b( X _{ [1/2,1]}),\ R\ae$, but $f$ fails to satisfy
	\begin{align*}
	f(X_0,X_1)=f_0(X_0)+f_1(X_1),
	\quad R\ae
	\end{align*}
	for some functions $f_0, f_1$ since this would imply that
	\begin{align*}
	f(x_1,z_1)&=f_0(x_1)+f_1(z_1)
	=f(x_1,z_2)-f_1(z_2)+f(x_2,z_1)-f_0(x_2)\\
	&=f(x_1,z_2)+f(x_2,z_1)-f(x_2,z_2)=0,
	\end{align*}
	a contradiction.
	
	\item
	\emph{A Markov measure $R$ not satisfying the assumption \ref{ass:formal_irreducibility}}.
	We consider the simple setting depicted at Figure \ref{fig-03} where all the drawn paths from left to right are  allowed. Note that  {$R$ can  be chosen as a Markov measure}. For instance with $R( \alpha_i \beta_j)= p_i q_j$ where $p_1=p_2=p_3=p_4=1/4,$ $q_1=1$ and $q_2=q_3=1/2.$ 
	It is assumed that all the states $x_1,\dots, z_3$ are distinct.  Clearly the assumption \ref{ass:formal_irreducibility} is not satisfied. Again, we exhibit a function $f$ satisfying \eqref{eq-19} but not \eqref{eq-08}.
	\\
	The function $f(X_0,X_1)$ is specified by $f(x_i,z_j)=f _{ ij}$ with
	$$
	f _{ 11}=2, f _{ 12}=2, f _{ 13}=3, f _{ 21}=2, f _{ 22}=3, f _{ 23}=4.
	$$   
	We see that $f( \alpha_i \beta_j)= a( \alpha_i) + b( \beta_j)$ where the functions $a$ and $b$ are given by:  
	$$a( \alpha_1)=1, a( \alpha_2)=0, a( \alpha_3)=1, a( \alpha_4)=1,   \quad \textrm{and} \quad b( \beta_1)=1,b( \beta_2)=2, b( \beta_3)=3.$$
	Denote $f_0(x_1)=p, f_0(x_2)=q, f_1(z_1)= r,f_1(z_2)=s,f_1(z_3)=t$ and suppose that \eqref{eq-08} holds, that is
	$$
	p+r=2,\quad q+r=2,\quad p+s=2,\quad q+s=3,\quad p+t=3,\quad q+t=4.
	$$
	The first two equations imply $p=q$. Plugging this into the third and fourth ones, leads to $2=3.$ 
	
	\begin{figure}
		\includegraphics{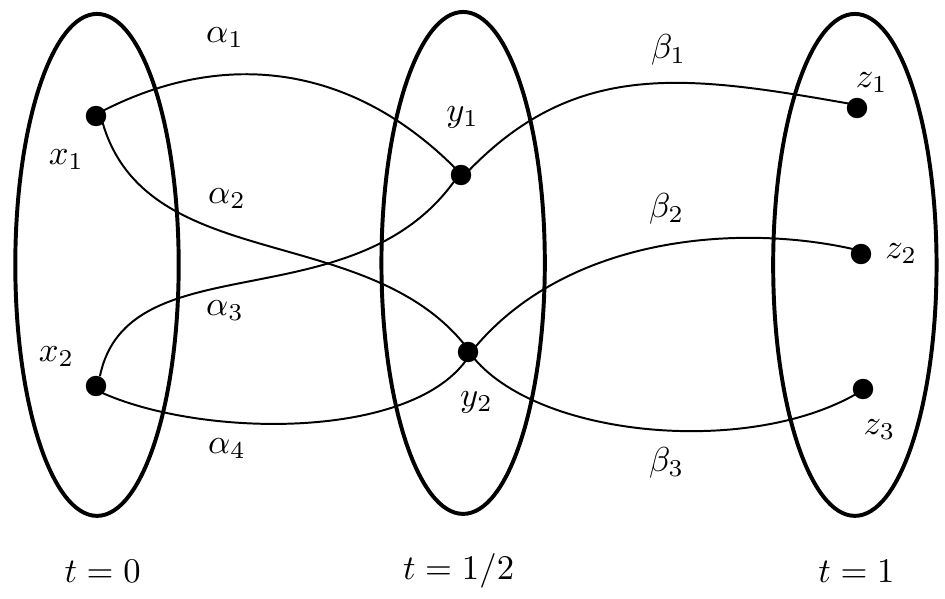}
		\caption{Another counter-example when assumption~\ref{ass:formal_irreducibility}. is dropped}\label{fig-03}
	\end{figure}
\end{enumerate}

The Assumptions \ref{ass:irreducibility} below are sufficient to prove that~\eqref{eq-19} implies~\eqref{eq-08}. They require the notion of reciprocal measure.

\subsection*{Reciprocal path measures}
Because of the nature of the minimization problems we  study, some of the processes are not  Markov in general. Still, they  satisfy some weaker independence properties, as they are reciprocal in the sense of the following definition:
\begin{definition}[Reciprocal measure] \label{def:reciprocal}
	A path measure  $Q$ on  $\Omega$ is called \emph{reciprocal} if it is conditionable and if it satisfies one of the two following equivalent assertions:
	\begin{itemize}
		\item For any times  $0< s< u< 1$  and  any events  $A\in \sigma(X_{[0,s]},X_{[u,1]})$ and $B\in \sigma(X_{[s,u]})$:
		\begin{equation}
		\label{eq:def_reciprocal}
		Q(A\cap B \mid X_s, X_u)=Q(A\mid X_s, X_u)Q(B\mid X_s, X_u)\quad  Q\ae . 
		\end{equation}
		\item For any times  $0< s< u< 1$  and  any events  $A\in \sigma(X_{[0,s]}) , B\in \sigma(X_{[s,u]}), C \in \sigma(X_{[u,1]}) $:
		\begin{equation}\label{eq-07}
		Q(A\cap B\cap C \mid X_s, X_u)=Q(A \cap C\mid X_s, X_u)Q(B\mid X_s, X_u)\quad  Q\ae  .
		\end{equation}
	\end{itemize}
\end{definition}

These properties  state that under $Q$, 
given the knowledge of the canonical process at  both times  $s$ and
$u$, the events \emph{inside} $[s,u]$ and those \emph{outside} $(s,u)$  are conditionally independent.
It is clearly  time-symmetric.

\begin{remarks} We recall basic relations between the Markov and reciprocal properties.
	\begin{enumerate}[(a)]
		\item
		\emph{Any Markov measure is reciprocal}. Indeed, let $Q$ be a Markov measure. Then,  for any $s,u,A,B$ and $C$ as in the second point of Definition~\ref{def:reciprocal}, 
		\begin{align*}
		&Q(A\cap B\cap C\mid X_s,X_u)=Q ^{ X_u}(A\cap B\cap C\mid X_s)
		=Q ^{ X_u}(A\mid X_s)Q ^{ X_u}(B\cap C\mid X_s)\\
		&\ =Q ^{ X_u}(A\mid X_s)Q ^{ X_s}(B\cap C\mid X_u)
		=Q ^{ X_u}(A\mid X_s)Q ^{ X_s}(B\mid X_u)Q ^{ X_s}(C\mid X_u)\\
		&\quad =Q ^{ X_u}(A\mid X_s)Q ^{ X_u}(C\mid X_s)Q ^{ X_s}(B\mid X_u)
		= Q(A\cap C\mid X_s,X_u)Q(B\mid X_s,X_u).
		\end{align*} 
		\item
		\emph{Conditionally to the initial value or the final value or both of them, a reciprocal measure is Markov}. Considering $s=0$ and $A=\OO$ in the definition of the reciprocal property~\eqref{eq-07}, we see that for all  $0\le u\le 1,$ $B\in \sigma(X _{ [0,u]})$ and $C\in \sigma (X _{ [u,1]}),$
		\begin{equation*}
		Q(B\cap C \mid X_0, X_u)=Q(B\mid X_0, X_u)Q(C\mid X_0, X_u)\quad  Q\ae  
		\end{equation*}
		This means that \emph{for any reciprocal measure $Q,$ the conditional path measure $Q ^{ X_0}:=Q(\sbt\mid X_0)$ is Markov, $Q\ae$}
		Similarly, with $u=1$ and $C=\OO$ in \eqref{eq-07}, we see that for all  $0\le s\le 1,$ $A\in \sigma(X _{ [0,s]})$ and $B\in \sigma (X _{ [s,1]}),$
		\begin{equation*}
		Q(A\cap B \mid X_s, X_1)=Q(A\mid X_s, X_1)Q(B\mid X_s, X_1)\quad  Q\ae  
		\end{equation*}
		meaning that \emph{the conditional path measure $Q ^{ X_1}:=Q(\sbt\mid X_1)$ is Markov, $Q\ae$}
		
		Because a Markov measure conditioned to its final (or initial) position is still Markov, one gets that if $Q$ is reciprocal, $Q^{X_0, X_1}$ is Markov $Q\ae$
	\end{enumerate}
\end{remarks}

\subsection*{Irreducible reciprocal measure}
We are ready to state our main assumption.

\begin{assumption}
	\label{ass:irreducibility}
Let $R \in \MO$ be reciprocal. We say it is irreducible if for all $0 \leq s < u < t \leq 1$, we have:
\begin{equation}
\label{eq:irreducibility}
R_s \otimes R_u \otimes R_t \ll R_{s,u,t}\ll R_s \otimes R_u \otimes R_t.
\end{equation}
(Recall that $R_{t_1, \dots, t_k}$ stands for the law of $(X_{t_1}, \dots, X_{t_k})$ under $R$.) In other terms, the laws $R_{s,u,t}$ and $R_s \otimes R_u \otimes R_t$ are equivalent in the sense of measure theory.
\end{assumption}
Let us prove that Assumption~\ref{ass:irreducibility} is sufficient to pass from~\eqref{eq-19} to~\eqref{eq-08}. We will then discuss a little bit further Assumption~\ref{ass:irreducibility}: we will see how it can be stated when $R$ is Markov, and how~\eqref{eq:irreducibility} tensorizes.

\begin{lemma}
	\label{lem:sum}
	Let $R \in \MO$ be reciprocal and satisfy Assumption~\ref{ass:irreducibility}. Suppose \eqref{eq-19} holds for some $0\leq s < u < t \leq 1$, and measurable $f$, $a$ and $b$. Then there exist measurable maps $f_s$ and $f_t$ such that~\eqref{eq-08} holds.
\end{lemma}
\begin{proof}
	Taking the conditional expectation in~\eqref{eq-19} with respect to $(X_s, X_u, X_t)$  leads to:
	\begin{align}
	\notag f(X_s, X_t) &= \EE_R[f(X_s, X_t)| X_s, X_u, X_t] \\
	\notag &= \EE_R[a(X_{[s,u]})|X_s, X_u, X_t] + \EE_R[b(X_{[u,t]})| X_s, X_u, X_t]\\
	\label{eq:conditioned} 
	&= \EE_R[a(X_{[s,u]})|X_s, X_u] + \EE_R[b(X_{[u,t]})| X_u, X_t],\qquad \qquad R\ae  
	\end{align}
	where we used the fact that $R$ is reciprocal to deduce the last line.
	\\
	Then, let us show that because of Assumption~\ref{ass:irreducibility}, formula~\eqref{eq:evbd_goes_to_y} holds.  For this, we use classical results in measure theory stating the behaviour of absolute continuity of measures with respect to conditioning. For the sake of completeness, these results are stated at Proposition~\ref{prop:conditioning} and Lemma~\ref{lem:CT}. Since $R_{s,u,t} \ll R_s \otimes R_u \otimes R_t$, by Proposition~\ref{prop:conditioning}, we have:
	\begin{equation*}
 R_{s,t} \ll 	R_s \otimes R_t.
	\end{equation*}
In addition,  by Lemma~\ref{lem:CT}, conditioning~\eqref{eq:irreducibility} on $X_u$ leads to
	\begin{equation*}
	\phantom{\qquad R\ae} R_s \otimes R_t \ll R((X_s, X_t) \in \sbt | X_u), \qquad R\ae
	\end{equation*}
	Bringing together the two last formulas, we get:
	\begin{equation*}
	 \phantom{\qquad R\ae} R_{s,t}\ll R((X_s, X_t) \in \sbt | X_u), \qquad R\ae,
	\end{equation*}
which is	\eqref{eq:evbd_goes_to_y}.
	Hence, let us pick some $y \in \XX$ such that:
	\begin{equation}
	\label{eq:AC}
	R_{s,t} \ll R((X_s, X_t) \in \sbt | X_u = y),
	\end{equation}
	such that $x \mapsto \EE[a(X_{[s,u]})|X_s = x , X_u = y]$ and $z \mapsto \EE[b(X_{[u,t]})| X_u =y , X_t = z]$ are well defined $R_s\ae$ and $R_t\ae$ respectively, and such that~\eqref{eq:conditioned} holds $R(\sbt|X_u = y)\ae$ (The set of such $y$'s has a full mass with respect to $R_u$.) We call:
	\begin{equation*}
	f_s(x) := \EE[a(X_{[s,u]})|X_s=x, X_u=y] \qquad \mbox{and} \qquad f_t(z) := \EE[b(X_{[u,t]})| X_u =y , X_t=z].
	\end{equation*}
	With this choice, by~\eqref{eq:conditioned}:
	\begin{equation*}
	f(x,z) = f_s(x) + f_t(z), \quad R((X_s, X_t) \in \sbt | X_u = y)\ae
	\end{equation*}
	Because of~\eqref{eq:AC}, it also holds $R_{s,t}\ae$ and the result follows.
\end{proof}

\subsection*{The Markov case}
In the case when $R$ is Markov, we can restrict ourselves to take a weaker assumption:
\begin{assumption}
	\label{ass:markov_irreducibility}
	Let $R \in \MO$ be Markov. We say it is irreducible if for all $0 \leq s < t \leq 1$, we have:
	\begin{equation}
	\label{eq:markov_irreducible}
	R_s \otimes R_t \ll R_{s,t} \ll R_s \otimes R_t.
	\end{equation}
\end{assumption}
Indeed, we have:
\begin{lemma}
	Let $R \in \MO$ be Markov and satisfy Assumption~\ref{ass:markov_irreducibility}. Then it satisfies also Assumption~\ref{ass:irreducibility}.
\end{lemma}
In particular, Lemma~\ref{lem:sum} holds for Markov processes satisfying only Assumption~\ref{lem:reciprocal_markov}.
\begin{proof}
	Let us take $0\leq s < u < t \leq 1$. First, conditioning $	R_u \otimes R_t \ll R_{u,t} \ll R_u \otimes R_t$ with respect to $X_u$  with the help of Lemma~\ref{lem:CT} leads to:
	\begin{equation}
	\label{eq:conditioned_irreducibility}
	\hspace{2cm}R_t \ll R(X_t \in \sbt | X_u) \ll R_t, \qquad R\ae 
	\end{equation}

	Then, we use the disintegration theorem to get the decomposition:
	\begin{equation*}
	R_{s,u,t} = R_{s,u} \otimes R(X_t \in \sbt | X_s, X_u) = R_{s,u} \otimes R(X_t\in \sbt | X_u),
	\end{equation*}
	where the second equality is obtained thanks to the Markov property of $R$. The result follows from combining~\eqref{eq:conditioned_irreducibility} and $R_s \otimes R_u \ll R_{s,u} \ll R_s \otimes R_u$.
\end{proof}

Next result shows that \eqref{eq:markov_irreducible} is indeed
an irreducibility requirement in the sense of Markov processes theory. Let us split it into:
\begin{align*}
&( \textrm{H}_1):\qquad R _{ st}\ll R_s\otimes R_t,\quad \forall 0\le s<t\le 1;\\
&( \textrm{H}_2):\qquad R_s\otimes R_t\ll R _{ st},\quad \forall 0\le s<t\le 1.
\end{align*}

\begin{proposition}
Let $R\in\MO$ be Markov.
\begin{enumerate}[(a)]
\item
Under the assumption $( \textrm{H}_1)$, for all $0\le s<t\le 1, $ we have
\begin{align}
R(X_t\in \bullet \mid X_s=x)&=r(s,x ;t,\bullet)\, R_t,\quad \forall x\in\XX,\ R_s\ae \label{eq-24b}\\
R(X_s\in \bullet \mid X_t=y)&=r(s,\bullet;t,y)\, R_s,\quad \forall y\in\XX,\ R_t\ae,
\notag
\end{align}
where the transition density is given by $ \displaystyle{r(s,x;t,y):= \frac{\D R _{ st}}{\D R_s\otimes R_t}(x,y)}.$

\item
If in addition  $( \textrm{H}_2)$ holds, i.e.\ \eqref{eq:markov_irreducible} is satisfied, then $r$ 
is positive in the sense that
\begin{align}\label{eq-24}
r(s,x;t,y)>0,\quad \forall (x,y), R_s\otimes R_t\ae,\quad \forall 0\le s<t\le 1.
\end{align}

\item
The property \eqref{eq:markov_irreducible} is equivalent to the existence of a transition density $r$ satisfying \eqref{eq-24b} and \eqref{eq-24}.
\end{enumerate}
\end{proposition}

\begin{proof}
Apply Lemma~\ref{lem:CT} to (H$_1$). We see that $R(X_s\in \bullet \mid X_t=y)\ll R_s$, for $R_t$-almost all $y$ and that $R(X_t\in \bullet\mid X_s=x)\ll R_t$, for $R_s$-almost all $x$. This implies that for all $0\le s<t\le 1, $ there are measurable functions $\overrightarrow{r}$ and $\overleftarrow{r}$ such that
\begin{align*}
R(X_t\in \bullet\mid X_s=x)&=\overrightarrow{r}(s,x;t,\bullet)\, R_t,\quad \forall x\in\XX,\ R_s\ae\\
R(X_s\in \bullet \mid X_t=y)&=\overleftarrow{r}(s,\bullet;t,y)\, R_s,\quad \forall y\in\XX,\ R_t\ae
\end{align*}
 It happens that 
\begin{align*}
\overrightarrow{r}(s,x;t,y)=\overleftarrow{r}(s,x;t,y)=r(s,x;t,y),
\quad \forall (x,y)\in\XXX,\ R _{ s}\otimes R_t\ae
\end{align*}
with $ \displaystyle{r(s,x;t,y):= \frac{\D R _{ st}}{\D R_s\otimes R_t}(x,y),}$ (again, we invoke (H$_1$)).
To see this, note that the joint measure $R _{ st}$  writes as 
\begin{alignat*}{3}
R _{ st}&= r(s,\bullet;t,\bullet)R_s \otimes R_t &&\\
	&=R_s \otimes R(X_t\in \bullet\mid X_s=x)
			 &\ =\ & \overrightarrow{r}(s,\bullet;t,\bullet) R_s \otimes R_t\\
		&=R_t \otimes R(X_s\in \bullet \mid X_t=y)
			 &\ =\ & \overleftarrow{r}(s,\bullet;t,\bullet) R_s\otimes R_t.
\end{alignat*}
(When there are two $\bullet$, the first one refers to the first variable, and the second one to the second variable.). Finally, under (H$_1$) the additional hypothesis (H$_2$) is equivalent to \eqref{eq-24} and statements (b) and (c) are obvious.
\end{proof}

The functions $\overrightarrow{r}$ and $\overleftarrow{r}$ are  the forward and backward transition densities.
By the first the statement of the proposition, it is correct to call $r$ the  transition density without mentioning any direction of time.

\subsection*{Tensorization of irreducibility}
Here we prove:
\begin{lemma}
	\label{lem:tensorization}
	Let $R \in \MO$ be a reciprocal process satisfying Assumption~\ref{ass:irreducibility}. Then for all $0 \leq t_1 < t_2 < \dots < t_k \leq 1$, we have:
	\begin{equation}
	\label{eq:tensorized_irreducibility}
	 R_{t_1} \otimes \dots \otimes R_{t_k} \ll R_{t_1,\dots, t_k} \ll R_{t_1} \otimes \dots \otimes R_{t_k}.
	\end{equation}
\end{lemma}
\begin{proof}
Let us prove this by induction. We suppose that~\eqref{eq:tensorized_irreducibility} holds for all $0 \leq t_1 < t_2 < \dots < t_k \leq 1$ for all $3\leq k \leq p$. Then, we choose $0 \leq t_1 < t_2 < \dots < t_{p+1} \leq 1$.  We use once again Lemma~\ref{lem:CT} to do the conditioning of
\begin{equation*}
R_{t_1}\otimes R_{t_p} \otimes R_{t_{p+1}} \ll R_{t_1, t_p, t_{p+1}} \ll R_{t_1}\otimes R_{t_p} \otimes R_{t_{p+1}}
\end{equation*}
with respect to $(X_{t_1}, X_{t_p})$. We get
\begin{equation}
\label{eq:conditioned_irreducibility_reciprocal}
\hspace{2cm} R_{t_1} \otimes R_{t_{p+1}} \ll R(X_{t_{p+1}}\in \sbt | X_{t_1}, X_{t_p}) \ll R_{t_1} \otimes R_{t_{p+1}}, \qquad R\ae
\end{equation}

Then, we use the disintegration theorem to get the decomposition:
\begin{equation*}
R_{t_1, \dots, t_{p+1}} = R_{t_1, \dots, t_p} \otimes R(X_{t_{p+1}} \in \sbt | X_{t_1},\dots, X_{t_p}) = R_{t_1, \dots, t_p} \otimes R(X_{t_{p+1}} \in \sbt | X_{t_1}, X_{t_p}),
\end{equation*}
where the second equality is obtained thanks to the reciprocity of $R$. The result follows from combining~\eqref{eq:conditioned_irreducibility_reciprocal} and the induction assumption.
\end{proof}

\section{Minimizing the relative entropy}\label{sec-min}

The purpose of this section is to prove Theorem~\ref{thm:sol_schro_additive} below, where we prove formula~\eqref{eq:general_factorization} for the solutions of the \emph{extended Schr\"odinger problem}, that is in the case when there is no endpoint constraint. Before stating and proving Theorem~\ref{thm:sol_schro_additive}, let us define precisely the relative entropy in the case when $R$ is unbounded. Then, we will state two basic properties of the relative entropy which will be useful in the proof of the theorem. Finally, we will introduce the extended Schr\"odinger problem, and give the result. 

The relative entropy with respect to the conditionable path  measure $R\in\MO$ is defined for all $P\in\PO$, by
\begin{equation*}
H(P|R):=\EE_P\left[\log \frac{\D P}{\D R} \right]\in (\infty, \infty].
\end{equation*}
In fact, this definition is not completely rigorous when $R$ is unbounded.

\subsection*{Relative entropy with respect to an unbounded measure}
\label{sec:unbounded}

If $R$ is unbounded, one must restrict the definition of $H(\sbt|R)$ to some subset of $\PO$ as follows. As $R$ is assumed to be  conditionable, it is a fortiori  $\sigma$-finite and  there exists some  measurable function $W:\OO\to [0,\infty)$ such that
\begin{equation}\label{eq-10}
z_W:=\IO e ^{-W}\D R<\infty.
\end{equation}
Define the probability measure $R_W:= z_W ^{-1}e ^{-W}\,R$ so that $\log(\D P/\D R)=\log(\D P/\D R_W)-W-\log z_W.$ It follows that for any $P\in \PO$ satisfying $\IO W\D P<\infty,$ the formula 
\begin{align*}
H(P|R):=H(P|R_W)-\IO W\D P&-\log z_W\in (-\infty,\infty],\\ 
&P\in \mathrm{P}_W(\OO):= \left\{P\in\PO: E_P [W]< \infty\right\}
\end{align*}
is a meaningful definition of the relative entropy which is coherent in the following sense. If $\IO W'\,\D P<\infty$ for another measurable function $W':\OO\to[0,\infty)$ such that $z_{W'}<\infty,$ then $H(P|R_W)-\IO W\D P-\log z_W=H(P|R _{W'})-\IO W'\D P-\log z_{W'}\in (-\infty,\infty]$.
\\
Therefore, $H(P|R)$ is well-defined for any $P\in \PO$ such that $\IO W\,\D P<\infty$ for some measurable nonnegative function $W$ verifying \eqref{eq-10}.

\subsection*{A  basic  lemma from statistical physics}
We recall one  fundamental easy result from statistical physics.
Let $(\zz_1, \rr_1)$ and $(\zz_2, \rr_2)$ be two measure spaces with $\rr_1$ and $\rr_2$ two probability measures on $\zz_1$ and $\zz_2$ respectively. Let us call $X_1,X_2$ the projections on the first and second variable respectively in the product space $\zz_1 \times \zz_2$. For any probability measure $\pi\in \mathrm{P}(\zz_1\times\zz_2)$ on the product space $\zz_1\times \zz_2,$ we denote $\pi_1:=X_1{}\pf \pi $ and $\pi_2:= X_2{}\pf\pi$ its marginals on $\zz_1$ and $\zz_2$ respectively.

\begin{lemma}\label{res-05}
For any $ \pi\in \mathrm{P}(\zz_1\times \zz_2)$, we have:
\begin{equation*}
H(\pi|\rr_1\otimes \rr_2)\ge H(\pi_1|\rr_1)+H(\pi_2|\rr_2)=H(\pi_1\otimes \pi_2|\rr_1\otimes \rr_2).
\end{equation*}
The corresponding equality: $H(\pi|\rr_1\otimes \rr_2)= H(\pi_1|\rr_1)+H(\pi_2|\rr_2),$ holds if and only if $\pi$ is a product measure, i.e.\ $\pi=\pi_1\otimes\pi_2.$
\end{lemma}

\begin{proof}
With the disintegration $\pi =\pi_1 \otimes \pi ^{ z_1}$ and the additive decomposition of the entropy, we see that
\begin{align*}
H(\pi|\rr_1\otimes \rr_2)&=H(\pi_1|\rr_1)+ \int _{ \zz_1}H(\pi ^{ z_1}\,|\, \rr_2)\, \pi_1(\D z_1)\\
	&\ge H(\pi_1|\rr_1)+H\Big( \int _{ \zz_1} \pi ^{ z_1}\, \pi_1(\D z_1)\,\Big|\, \rr_2\Big)\\&=H(\pi_1|\rr_1)+H(\pi_2|\rr_2)=H(\pi_1\otimes \pi_2|\rr_1\otimes \rr_2),
\end{align*}
where the inequality is a consequence of the convexity of the relative entropy and Jensen's inequality. Tracking the equality in Jensen's inequality, we see that $\pi$ must satisfy $\pi ^{ z_1}=\pi_2$ for $\pi_1$-almost all $z_1\in\zz_1,$ i.e.\ $\pi=\pi_1\otimes \pi_2.$
\end{proof}

Let us recall the physical interpretation of this result. Consider two \emph{noninteracting} random particle systems  \textsf{1}  and \textsf{2} respectively governed by the measures $\rr_1$ and $\rr_2$. Because of the absence of interaction, the whole system \textsf{1}+\textsf{2} is governed by the product measure $\rr_1\otimes\rr_2$.  Suppose that one observes that the average configurations of \textsf{1} and \textsf{2} are respectively close to $\pi_1$ and $\pi_2$. Then, the most likely actual state $\pi_{ \textsf{1}+\textsf{2}}$ of the whole system \textsf{1}+\textsf{2} is the product $\pi_1\otimes \pi_2,$ meaning that no extra correlation should come into the picture. This result is often quoted as the \emph{additivity of the entropy of noninteracting systems}, since $H(\pi _{ \textsf{1}+\textsf{2}}|\rr_{ \textsf{1}+\textsf{2}})=H(\pi_1|\rr_1)+H(\pi_2|\rr_2).$

\subsection*{A  basic  lemma from theoretical statistics}

Next result is well-known  in theoretical statistics where it gives rise to the notion of \emph{exhaustive statistics} of a dominated statistical model, see~\cite{dacunha2012probability}. Let $\zz$ and $\SS$ be two Polish spaces equipped with their Borel $ \sigma$-fields and $ S:\zz\to\SS$ a measurable mapping. Consider a positive $ \sigma$-finite measure $\rr$ and a probability measure $\pp$  on $\zz$ such that $\pp\ll \rr$. We denote $\rr _{ S}:= S\pf \rr$ and $\pp_{ S}:= S\pf \pp$ and we assume that  $\rr _{ S}$ is also a $ \sigma$-finite measure on $\SS$ to be able to consider the conditional measure $\rr(\sbt| S)$. Because of Proposition~\ref{prop:conditioning}, we have $\pp_s \ll \rr_S$. 

\begin{lemma}\label{res-06}
Let $\rr,\pp$ and $ S$ as above. The two following assertions are equivalent:
\begin{enumerate}[(i)]
\item the following conditional laws coincide:
\begin{align*}
\pp(\sbt| S)=\rr(\sbt | S),\quad \pp\ae
\end{align*}
\item There is a measurable map $f\SS\to [0, \infty)$ such that:
\begin{align*}
\frac{\D \pp}{\D \rr}=f\circ S,\ \rr\ae
\end{align*}
\end{enumerate}
In this case, we have $ \displaystyle{f= \frac{\D \pp _{ S}}{\D \rr _{ S}}}$.
\end{lemma}

\begin{proof}
\Boulette{(i)$\Rightarrow$(ii)}
The statement: "$\pp(\sbt| S)=\rr(\sbt | S),\ \pp\ae$" is equivalent to $\pp=\int _{ \SS}\rr(\sbt| S=s)\, \pp _{ S}(\D s)$. Hence, for any bounded measurable function $u:\zz\to\RR,$
\begin{align*}
\EE_\pp[u]&=\int _{ \SS}\EE_\rr[u| S=s]\, \pp _{ S}(\D s)
	=\int _{ \SS}\EE_\rr[u| S=s] \frac{\D \pp _{ S}}{\D \rr _{ S}}(s)\, \rr _{ S}(\D s)\\
	&=\int _{ \SS}\EE_\rr\Big[\frac{\D \pp _{ S}}{\D \rr _{ S}}(s)\ u \Big| S=s\Big] \, \rr _{ S}(\D s)
	= \int _{ \zz}\frac{\D \pp _{ S}}{\D \rr _{ S}}( S)\ u\D \rr
\end{align*}
which means $ \displaystyle{ \frac{\D \pp}{\D \rr}}=\frac{\D \pp _{ S}}{\D \rr _{ S}}( S).$

\Boulette{(ii)$\Rightarrow$(i)} For any bounded measurable function $u:\zz\to\RR,$
\begin{equation*}
\EE_\pp[u]=\int _{ \zz}f( S)\ u\D \rr
	=\int _{ \SS}\EE_\rr[u| S=s] f(s)\, \rr _{ S}(\D s)
	=\int _{ \SS}\EE_\rr[u| S=s]\, (f\rr _{ S})(\D s).
\end{equation*}
Choosing $u=v\circ S$ with $v:\SS\to\RR$ bounded and measurable leads to
$\EE_{ \pp_{ S}}(v)=\int _{ \SS}v\, \mathrm{d}(f\rr _{ S})$. This means that $ \displaystyle{f= \frac{\D \pp _{ S}}{\D \rr _{ S}}}.$ Finally, identifying 
$ \EE_\pp[u\ v( S)]=\int _{ \SS}\EE_\pp[u| S=s]v(s)\, \pp_{ S}(\D s)$ with $\EE_\pp[u\ v( S)]=\int _{ \SS}\EE_\rr[u| S=s]v(s)\, \pp_{ S}(\D s)$ gives us 
$\pp(\sbt| S)=\rr(\sbt | S),\ \pp\ae$
\end{proof}

\subsection*{Statement of the extended Schrödinger problem}

We are interested in the marginal constraint $P_t= \mu_t$ for all $t\in\TT\subset [0,1]$, where $ \mu_t\in\PX, \ t\in\TT.$ The entropy minimization problem we consider is:
\begin{equation}\label{eq-11}
H(P|R)\to \textrm{min};\qquad P\in\PO: P_t= \mu_t,\forall t\in\TT.
\end{equation}
It is a generalization of the dynamical Schrödinger problem corresponding to $\TT= \left\{0,1\right\}$. The properties of the relative entropy lead to the following existence result stating that a unique solution to this problem exists if and only if there is at least one competitor. Results concerning the existence of such competitors can be found for instance in~\cite{daw87,follmer1988random}, but we do not wish to develop them here.

\begin{proposition}\label{res-10}
Problem  \eqref{eq-11} admits a solution if and only if there exists some $Q\in\PO$ such that $Q_t= \mu_t$ for all $t\in\TT$ and $H(Q|R)< \infty.$ In this case, the solution $P$ is unique.
\end{proposition}

\noindent\textit{Sketch of proof.}\ 
This argument is standard. The conclusion follows  from the three following facts:
\begin{enumerate}[(i)]
\item
As $\Omega$ is a Polish space, the relative entropy $H(\sbt|R)$ has compact sublevels on $\PO$ with respect to the usual narrow topology $\sigma(\PO,\CO)$.
\item
The constraint set $ \left\{P\in\PO; P_t= \mu_t,\ t\in\TT\right\} $ is closed. 
\item
The existence of $Q$ in our assumption implies that the compact set 
\begin{equation*}
\left\{P\in\PO; P_t= \mu_t,\ t\in\TT\} \cap \{P\in\PO; H(P | R) \leq H(Q|R) \right\}
\end{equation*} 
is not empty. Hence, it contains a solution.
\item Uniqueness follows from the strict convexity of the relative entropy, and from the convexity of the constraint set.
\end{enumerate}

\begin{remark}[The closure of $\TT$]
For all $P \in \PO$, $t \mapsto P_t$ is  continuous.  For this reason, the extended Schrödinger problem~\eqref{eq-11} only admits solution if $t\in \TT \mapsto \mu_t$ can be extended into a continuous map on $[0,1]$. In particular, it must admit limits for $t \in \partial \TT$, and if so, for all $t \in \partial \TT$, the property:
\begin{equation*}
P_t = \lim_{\substack{s \to t \\ s \in \TT}} \mu_s
\end{equation*}
holds automatically.

Hence, it is clear that we can suppose without loss of generality that $\TT$ is closed. This is  systematically assumed in the following.
\end{remark}

\subsection*{Factorization result for the extended Schr\"odinger problem} We are now ready to state the central result of this article: a solution to this type of extended Schrödinger problems is Markov, and the  { additive functional} given by Theorem~\ref{res-02} cancels outside $\mathcal{T}$. Even if in next Section \ref{sec:finitely-many} we  only apply this result to the case when $\TT$ is finite, we think that this result is interesting in the general case.

\begin{theorem}	\label{thm:sol_schro_additive}
	Suppose that $R$ is Markov, and that problem~\eqref{eq-11} admits a (unique) solution $P$. Then $P$ is Markov. In particular, there exists  a regular additive functional $A$ such that:
	\begin{equation}
	\label{eq:sol_schro_additive}
	\frac{\D P}{\D R} = \exp\Big( A\big( [0,1] \big) \Big).
	\end{equation}
	Moreover, $A([0,1])$ is $ \sigma(X_\TT)$-measurable, and for any interval $I \subset[0,1]$ with $\inf I = s$ and $\sup I = t$, $A(I)$ is $\sigma(X_s, X_t, X_{I \cap \TT})$-measurable.
	
	Besides, if $\TT$ is closed, and if $R$ satisfies the irreducibility Assumption~\ref{ass:markov_irreducibility}, there exists an additive functional $A$ (which is not regular anymore; see Remark~\ref{rem:lost_of_regularity}) with the following additional property: for all interval $I \subset [0,1]$:
	\begin{equation}
	\label{eq:A_cancellations}
	I \cap \TT = \emptyset \quad \Rightarrow \quad A( I ) = 0.
	\end{equation}
	In this case, for all interval $I \subset[0,1]$, $A(I)$ is $X_{I \cap \TT}$-measurable.
	
	If the complementary of $\TT$ is a finite union of disjoint open intervals, then this additive functional $A$ is regular.
\end{theorem}

\begin{remark}
	\label{rem:lost_of_regularity}
	As we saw in the statement of Theorem~\ref{thm:sol_schro_additive}, for a general $\TT$ we do not know how to access to  a regular additive functional satisfying~\eqref{eq:A_cancellations}. However, the $A$ we are going to build is in some sense "regular up to a closed set of empty interior". More precisely if one writes the complementary of $\TT$ as a countable union of disjoint open intervals:
	\begin{equation*}
	\TT^c = \bigsqcup_ {i\in\Lambda} (s_i, t_i),
	\end{equation*}
	then regularity fails at the accumulation points of the endpoints $\{s_i, t_i\, | \, i \in \Lambda\}$. In particular, if the union is finite, regularity is preserved, as stated in Theorem~\ref{thm:sol_schro_additive}.
\end{remark}

\begin{proof}
\noindent$\bullet$ \ \emph{The solution $P$ is Markov.}
We denote for any $0< s< 1$, $Q ^{ X_s}:=Q(\sbt| X_s),$ $Q ^{ X_s}_{ \leftarrow}:=(X _{ [0,s]})\pf Q ^{ X_s}$ and $Q ^{ X_s}_{ \rightarrow}:=(X _{ [s,1]})\pf Q ^{ X_s}$ for any conditionable $Q\in\MO.$ Let us also define for any $Q\in\PO,$
\begin{equation*}
\widetilde Q^s:= \IX [Q ^{ X_s=x}_{ \leftarrow}\otimes Q ^{ X_s=x}_{ \rightarrow}] \, Q_s(\D x)\in\PO.
\end{equation*}
Referring to \eqref{eq-01}, we have to prove that for all $0<s<1,$ the solution $P$ satisfies
\begin{equation*}
P^{ X_s}=P ^{ X_s} _{ \leftarrow}\otimes P ^{ X_s}_{ \rightarrow},\quad P\ae
\end{equation*}
or equivalently
\begin{align*}
\widetilde P^s=P.
\end{align*}
To this purpose, it is sufficient to show that
\begin{enumerate}[(i)]
\item
for any  $Q\in\PO,$ we have
\begin{align*}
H(\widetilde Q^s|R)\le H(Q|R),\quad \forall 0< s<1
\end{align*}
with equality if and only if $Q=\widetilde Q^s,$ 
\item
and also that
\begin{align*}
\widetilde Q_t^s=Q_t\in\PX,\quad \forall 0\le t\le1,\ 0<s<1.
\end{align*}
\end{enumerate}
Let us fix $s$ and $Q\in\PO$ and start proving (i). By the additive decomposition of  the entropy, 
\begin{align*}
H(Q|R)&=H(Q_s|R_s)+\IX H(Q ^{ X_s=x}\,|\, R ^{ X_s=x})\, Q_s(\D x)\\
	&=H(Q_s|R_s)+\IX H(Q ^{ X_s=x}\,|\, R ^{ X_s=x}_{ \leftarrow}\otimes R ^{ X_s=x}_{ \rightarrow})\, Q_s(\D x)\\
	&\ge H(Q_s|R_s)+\IX H(Q ^{ X_s=x}_{ \leftarrow}\otimes Q ^{ X_s=x}_{ \rightarrow}\,|\, R ^{ X_s=x}_{ \leftarrow}\otimes R ^{ X_s=x}_{ \rightarrow})\, Q_s(\D x),
\end{align*}
where at the second equality we use the assumed Markov property of $R$ and at the last inequality we apply Lemma \ref{res-05} with $\rr_1=R ^{ X_s=x}_{ \leftarrow}$ and $\rr_2=R ^{ X_s=x}_{ \rightarrow}.$ In addition, Lemma~\ref{res-05} tells us that if the identity:
\begin{equation*}
Q^{ X_s}=Q ^{ X_s} _{ \leftarrow}\otimes Q ^{ X_s}_{ \rightarrow},\quad Q\ae
\end{equation*}
 fails, then the last inequality is strict. This proves (i).

 It remains to prove (ii), i.e.\  the operation
$Q \leadsto \widetilde Q^s$ does not alter the marginals.
For any bounded measurable function $a:\zz\to\RR$ and any $t\in[0,s],$
\begin{align*}
E_Q[a(X_t)]
	=\IX E _{ Q ^{ X_s=x}_{ \leftarrow}}[a(X_t)]\, Q_s(\D x)
	=\IX E _{ Q ^{ X_s=x}_{ \leftarrow}\otimes Q ^{ X_s=x}_{ \rightarrow}}[a(X_t)]\, Q_s(\D x)=E _{ \widetilde Q^s}[a(X_t)]
\end{align*}
and for any $t\in[s,1]$ a similar reasoning still works:
\begin{align*}
E_Q[a(X_t)]
	=\IX E _{ Q ^{ X_s=x}_{ \rightarrow}}[a(X_t)]\, Q_s(\D x)
	=\IX E _{ Q ^{ X_s=x}_{ \leftarrow}\otimes Q ^{ X_s=x}_{ \rightarrow}}[a(X_t)]\, Q_s(\D x)=E _{ \widetilde Q^s}[a(X_t)].
\end{align*}
This proves (ii).
\\
Therefore, the solution $P$  is Markov and we know with Theorem \ref{res-02} that this is equivalent to:
\begin{align*}
\frac{\D P}{\D R}=\exp\Big( A\big( [0,1] \big) \Big)
\end{align*}
for some regular additive functional $A$.
\\
\noindent$\bullet$ \ \emph{Measurability.}
To prove that $A\big( [0,1] \big)$ is $ \sigma(X_\TT)$-measurable, it is enough to show that
\begin{equation}\label{eq-13}
\frac{\D P}{\D R} \mbox{ is }X_{\TT}\mbox{-measurable}.
\end{equation}
Using the additive decomposition of the entropy, we see that
\begin{align*}
H(P|R)=H(P _{ \TT}|R _{ \TT})
	+\int _{ \zz ^{ \TT}}H(P ^{ X _{ \TT}=\zeta}| R ^{ X _{ \TT}=\zeta})\,P _{ \TT}(\D \zeta)
\end{align*}
and it follows that 
\begin{equation*}
\widehat P := \int _{ \zz ^{ \TT}} R ^{ X _{ \TT}=\zeta} \,P _{ \TT}(\D \zeta)\in\PO
\end{equation*}
satisfies
\begin{align*}
H(\widehat P|R)=H(P _{ \TT}|R _{ \TT})\le H(P|R).
\end{align*}
We also have  $\widehat P _{ \TT}=P _{ \TT}$,  implying a fortiori that $\widehat P_t=P_t$ for all $t\in\TT.$ As a consequence, the solution $P$ of the Schrödinger problem satisfies $P=\widehat P.$  In other words $P ^{ X _{ \TT}}= R ^{ X _{ \TT}},$ $P\ae$ and we know with Lemma \ref{res-06} that this is equivalent to \eqref{eq-13}. 
\\
The proof of the fact that for an interval $I \subset [0,1]$ with $\inf I = s$ and $\sup I = t$, $A(I)$ is $\sigma(X_s, X_t, X_{I \cap \TT})$-measurable follows the same lines by noticing that the restriction $X_I {}\pf P$ of $P$ to $I$  is a solution of the extended Schrödinger problem~\eqref{eq-11} between the times $s$ and $t$:
\begin{equation*}
H(Q|X_{\overline I}{}\pf R)\to \textrm{min};\quad P\in\mathrm{P}\big(\Omega_{\overline{I}}): Q_u= \mu_u,\forall u\in I\cap\TT, \, Q_s = P_s, \, Q_t = P_t,
\end{equation*}
where $\Omega_I$ is the set of paths  on the time interval $\overline{I}$, with values in $\XX$.
\\
\noindent$\bullet$ \ \emph{Construction of a version of $A([0,1])$  canceling outside $\TT$.} 
Now, in case $R$ satisfies the irreducibility Assumption~\ref{ass:markov_irreducibility}, the goal is to build an additive functional $\tilde{A}$ such that $\tilde{A}([0,1]) = A([0,1])$  and such that~\eqref{eq:A_cancellations} holds for $\tilde{A}$.

First, we can suppose without loss of generality that $0,1 \in \TT$. Indeed, if it is not the case, by calling $s := \min \TT$ and $t := \max \TT$, it suffices to prove the result for the restriction $P_{[s,t]}$ and $R_{[s,t]}$ and then to extend the obtained additive functional by $0$ on the intervals included outside $[s,t]$.  

Here is the procedure to define $\tilde{A}$ from $A$. Let us write the complementary set of $\TT$ as a countable union of disjoint open intervals:
\begin{equation*}
\TT^c = \bigsqcup_ {i\in\Lambda} (s_i, t_i).
\end{equation*}
For each $i \in \Lambda$, call $u_i := (s_i + t_i)/2$. We have:
\begin{equation*}
A((s_i,t_i)) = A((s_i,u_i]) + A((u_i,t_i)).
\end{equation*}
but by the first part of the theorem, $A((s_i,t_i))$ is $\sigma(X_{s_i}, X_{t_i})$-measurable, so that we are in the framework of Lemma~\ref{lem:sum}: there exist $\alpha_i, \beta_i$ such that $R\ae$,
\begin{equation*}
A((s_i,t_i)) = \alpha_i(X_{s_i}) + \beta_i(X_{t_i}).
\end{equation*}

Let us define $\tilde{A}([s,t])$ for $s \leq t \in [0,1]$ in the following way. We call $ \varphi(s) := t_i$ if $s_i \leq s \leq t_i$ and $\varphi(s) = s$ otherwise. Correspondingly, we call $\psi(t)=  s_i$ if $s_i \leq t \leq t_i$ and $\psi(t) = t$ otherwise. Then, we set: 
\begin{equation*}
\tilde{A}\big( [s,t] \big) := \1_{\{\varphi(s) \leq \psi(t)\}} A\big( [\varphi(s),\psi(t)] \big) + \sum_{i \in \Lambda} \1_{\{ s_i <s \leq t_i \leq t\}} \beta_i(X_{t_i}) + \sum_{j \in \Lambda} \1_{\{ s \leq s_j\leq t <t_j\}} \alpha_j(X_{s_j}).
\end{equation*}
(Note that for each sum, at most one term is nonzero so that this formula is well defined.) Clearly  $\tilde{A}([0,1]) = A([0,1])$,  $\tilde{A}([s,t])$ is  $X_{[s,t]}$-measurable, and for any interval $I \subset\TT^c$, this formula leads to $\tilde{A}(I) = 0$. 

Let us prove that $\tilde{A}$ is an additive functional. Once again, we follow Remark~\ref{contentonclosedintervals}. Let us take $0 \leq s  \leq u \leq v \leq t$ and  suppose that $\tilde{A}([u,v]) = - \infty$. We show that:
\begin{equation}
\label{eq:condition_infinity}
\tilde{A}([s,v]) = \tilde{A}([u,t]) = \tilde{A}([s,t]) = -\infty.
\end{equation}
It means that $\varphi(u) \leq \psi(v)$ and that one of the following holds:
\begin{itemize}
	\item $A([\varphi(u), \psi(v)]) = - \infty$,
	\item there exists $i \in \Lambda$ such that $s_i < u \leq t_i \leq v$ and $\beta_i (X_{t_i}) = - \infty$,
	\item there exists $i \in \Lambda$ such that $u \leq s_i \leq v < t_i$ and $\alpha_i (X_{s_i}) = - \infty$.
\end{itemize}
In the first case, as $[\varphi(u), \psi(v)]$ is a subset of $[\varphi(s), \psi(v)]$, $[\varphi(u), \psi(t)]$ and $[\varphi(s), \psi(t)]$, $\tilde{A} = - \infty$ on this three intervals, so that~\eqref{eq:condition_infinity} holds. 
\\
In the second case, $\beta_i(X_{t_i})$ also intervenes in the definition of $\tilde{A}([u,t])$, which is hence infinite. Then, either $s_i< s$, and so $\beta_i(X_{t_i})$ also intervenes in the definition of $\tilde{A}([s,v])$, and $\tilde{A}([s,t])$ which let us conclude, or $s \leq s_i$ so that $(s_i,t_i)$ is a subset of both $[\varphi(s), \psi(v)]$ and $[\varphi(s), \psi(t)]$. In this last case, we conclude by using the fact that thanks to~\eqref{eq:condition_infinity}, $A((s_i, t_i)) = - \infty$.
\\
The third case is treated in the same way.

Now, we suppose that $\tilde{A}([u,v])$ is finite and we want to show:
\begin{equation}
\label{eq:check_additive}
\tilde{A}([s,t]) + \tilde{A}([u,v]) = \tilde{A}([s,v]) + \tilde{A}([u,t]).
\end{equation}
There are several cases to deal with, let us treat them one by one.
\begin{itemize}
	\item If there is some  $i \in \Lambda$ such that $[s,t] \subset (s_i, t_i)$, then every term in~\eqref{eq:check_additive} is zero.
	\item If $s,u$ are in $\TT \backslash ( \cup_i (s_i,t_i])$, and $v,t$ are in $\TT \backslash( \cup_i [s_i,t_i))$ then~\eqref{eq:check_additive} is a consequence for the same formula with $A$ instead of $\tilde{A}$.
	\item If for instance for some $i \in \Lambda$, $s_i < s \leq t_i \leq v $, then $\beta_i (X_{t_i})$ appears once on the left-hand side of~\eqref{eq:check_additive}, in the definition of $\tilde{A}([s,t])$, and once on the right-hand side, in the definition of $\tilde{A}([s,v])$. An analogous argument allows  us to treat the cases where $u \leq s_i \leq t < t_i$, $s_i < u \leq t_i \leq v \leq t$ and $u \leq s_i \leq v < t_i$.
	\item If there is some $i \in \Lambda$ such that $[u,v] \subset (s_i,t_i)$, $s \leq s_i$ and $t \geq t_i$. In that case, $\tilde{A}([u,v]) = 0$, and
	\begin{align*}
	A([\varphi(s), \psi(t)]) &= A([\varphi(s), s_i]) + A((s_i,t_i)) + A([t_i, \psi(t)]) \\
	&= A([\varphi(s), s_i]) + \alpha_i(X_{s_i}) + \beta_i(X_{t_i}) + A([t_i, \psi(t)]).
	\end{align*}
	The result follows easily by taking into consideration the endpoint terms for $s$ and $t$ thanks to the previous points.
	\\
	The cases when $[s,v] \subset (s_i,t_i) $ or $[u,t] \subset (s_i,t_i)$ are similar and left to the reader.
\end{itemize}

Finally, for this choice of $\tilde A$, if $I \subset [0,1]$ is an interval with $\inf I = s$ and $\sup I = t$, then calling $\tilde{s} := \inf I\cap \TT$, $\tilde{t} := \sup I \cap \TT$ and $\tilde{I} := [\tilde{s}, \tilde{t}]$, then, as $I \backslash \tilde{I}$ is the union of at most two intervals which do not intersect $\TT$:
\begin{equation*}
\tilde A(I) = \tilde A(\tilde{I}) \quad \mbox{is }X_{I \cap \TT}\mbox{-measurable}.
\end{equation*}

The regularity of $\tilde{A}$ in the case when $\Lambda$ is finite is a consequence of the fact that then, for each $i,s$, $t \mapsto \1_{\{ s \leq s_i\leq t <t_i\}} \alpha_i(X_{s_i})$ is right continuous and left limited, and for all $i,t$ $s \mapsto \1_{\{s_i < s \leq t_i \leq t\}} \beta_i(X_{t_i})$ is left continuous and right limited.
\end{proof}

\section{Under finitely many marginal constraints}
\label{sec:finitely-many}
\subsection*{The Schrödinger case}
Let us consider the easiest setting where there are finitely many constraints, i.e.\ $\TT= \left\{t_1,t_2, \dots,t_K\right\}$ and $R$ is Markov and irreducible in the sense of Assumption~\ref{ass:markov_irreducibility}.
Applying Theorem~\ref{thm:sol_schro_additive}, we obtain

\begin{theorem}
\label{thm:discrete_schro}
Suppose $R$ is Markov, irreducible in the sense of Assumption~\ref{ass:markov_irreducibility} and that $\TT := \{ t_1, \dots, t_K \}$ is finite.
If the extended Schrödinger problem~\eqref{eq-11} admits a (unique) solution, then the corresponding additive functional $A([0,1])$ given by Theorem~\ref{thm:sol_schro_additive} writes as:
\begin{equation*}
A\big([0,1]\big)
	=\sum _{i=1}^K f_i(X _{ t_i})
\end{equation*}
for some measurable functions $f_i$, $1\le i\le K$. 
\end{theorem}
\begin{proof}
	Let us decompose:
	\begin{equation*}
	A\big([0,1]\big) = A\big( [0,t_1) \big) + A\big( (t_K,1 ] \big) + \sum_{i=1}^K A\big( \{ t_i \} \big) + \sum_{i=1}^{K-1} A\big( (t_i, t_{i+1}) \big) .
	\end{equation*}
	Each term of the form $A( [0,t_1))$,  $A( (t_K,1 ])$ or $A( (t_i, t_{i+1}))$ cancels because of~\eqref{eq:A_cancellations}. Then, each term of the form $A( \{ t_i \} )$ is $X_{t_i}$-measurable by the measurability property of an additive functional, and  hence of the form $f_i(X_{t_i})$.
\end{proof}

\subsection*{The Brödinger case}

We are now interested in adding another marginal constraint to the Schrödinger problem \eqref{eq-11}.  The Brödinger entropy minimization problem  is 
\begin{equation}\label{eq:pb_bro}
H(P|R)\to \textrm{min};\qquad P\in\PO: P_t= \mu_t,\forall t\in\TT,\ P _{ 01}=\pi
\end{equation}
where $ \pi\in\PXX$ is a prescribed  endpoint marginal.

\begin{proposition}
	Problem \eqref{eq:pb_bro} admits a solution if and only if there exists some $Q\in\PO$ such that $Q_t= \mu_t$ for all $t\in\TT,$ $Q _{ 01}=\pi$ and $H(Q|R)< \infty.$ In this case, the solution $P$ is unique.
\end{proposition}

\begin{proof}
	It follows the same line as the proof of Proposition \ref{res-10}.
\end{proof}

However, the solutions of this type of problems are not Markov in general, even when $R$ is Markov. Instead, they are reciprocal in the sense of Definition~\ref{def:reciprocal}. Before proving this and giving the analogue of Theorem~\ref{thm:sol_schro_additive} in this context, let us state and prove another link between Markov and reciprocal measures which will permit us to apply Theorem~\ref{thm:discrete_schro} in the Brödinger case.
\begin{lemma}
\label{lem:reciprocal_markov}
	Let $Q \in \MO$ be conditionable. For $\lambda \in (0,1)$, let us define $\Psi_\lambda$ by:
	\begin{equation}
	\label{eq:def_Psi_lambda}
	\begin{aligned}
	\Psi_{\lambda}:&& C^0([0,1]; \XX) &\longrightarrow C^0([0,1]; \XX \times \XX), \\
	&& \Big(t\mapsto \omega_t\Big) &\longmapsto \Big(t \mapsto (\omega_{\lambda t}, \omega_{1 - (1-\lambda) t})\Big).
 	\end{aligned}
	\end{equation}
	Then $Q$ is reciprocal if and only if for all $\lambda \in (0,1)$, $Q_\lambda := \Psi_{\lambda} {}\pf Q$ is Markov.
\end{lemma}
\begin{proof}
	In this proof, we call $Z_t = (Z^1_t, Z^2_t)$ the canonical process on $C^0([0,1]; \XX \times \XX)$ at time $t$, and we keep on using the notation $X_t$ to denote the canonical process on $C^0([0,1]; \XX)$ at time $t$. Of course, the Definition~\ref{def:markov} of a Markov process needs to be adapted to processes with values in $\XX \times \XX$.
	
	Let us suppose that $Q$ is reciprocal, and let us take $\lambda \in (0,1)$. As $\Psi_\lambda$ is injective, the path measure $Q_\lambda$ is conditionable. Let us take $t\in[0,1]$, $A\in \sigma(Z_{[0,t]})$ and $B\in \sigma(Z_{[t,1]})$. We have:
	\begin{align*}
	Q_\lambda(A \cap B | Z_t) &= Q\Big( \Psi_\lambda^{-1}\big( A \cap B \big) \Big| X_{\lambda t} , X_{1 - (1-\lambda)t}\Big)\\
	&= Q\Big( \Psi_\lambda^{-1}( A) \cap \Psi_\lambda^{-1}( B ) \Big| X_{\lambda t} , X_{1 - (1-\lambda)t}\Big).
	\end{align*}
	But it is clear that $\Psi_\lambda^{-1}( A)$ is in $\sigma(X_{[0,\lambda t]}, X_{[1 - (1-\lambda) t, 1]})$ and $\Psi_\lambda^{-1}(B)$ is in $\sigma(X_{[\lambda t , 1 - (1-\lambda) t]})$, so that as $Q$ is reciprocal, by~\eqref{eq:def_reciprocal}:
	\begin{align*}
	Q_\lambda(A \cap B | Z_t) &= Q\Big( \Psi_\lambda^{-1}( A) \Big| X_{\lambda t} , X_{1 - (1-\lambda)t} \Big)Q\Big( \Psi_\lambda^{-1}( B ) \Big| X_{\lambda t} , X_{1 - (1-\lambda)t}\Big)\\
	&= Q_\lambda(A|Z_t) Q_\lambda(B|Z_t).
	\end{align*}
	Hence, $Q_\lambda$ is Markov.
	
	Now, let us suppose that for all $\lambda \in (0,1)$, $Q_\lambda$ is Markov. Let us fix $0 < s < u < 1$, $A \in\sigma(X_{[0,s]}, X_{[u,1]})$ and $B \in \sigma(X_{[s,u]})$. We choose $\lambda := s /(1 - (u-s))$, so that for:
	 \begin{equation*}
	 \bar{t} := \frac{s}{\lambda} = \frac{1-u}{1-\lambda} = 1 - (u-s),
	 \end{equation*}
we have both $\lambda \bar{t} = s$ and $1 - (1-\lambda)\bar{t} = u$. Then, as $\Psi_\lambda$ is injective, we have:
\begin{align*}
Q(A \cap B | X_s,X_u) &= Q\Big( \Psi_\lambda^{-1}\big( \Psi_\lambda(A) \big) \cap \Psi_\lambda^{-1}\big( \Psi_\lambda(B) \big) \Big| Z_{\bar{t}}\Big)\\
&= Q_\lambda\big( \Psi_\lambda(A) \cap \Psi_\lambda(B)  \big| Z_{\bar{t}} \big).
\end{align*}
Noticing that $\Psi_\lambda(A) \in \sigma (Z_{[0, \bar{t}]})$ and $\Psi_\lambda(B) \in \sigma(Z_{[\bar{t}, 1]})$ and using the Markov property of $Q_\lambda$, we see that
\begin{align*}
Q(A \cap B | X_s,X_u) &= Q_\lambda\big( \Psi_\lambda(A) \big| Z_{\bar{t}} \big)Q_\lambda\big( \Psi_\lambda(B)  \big| Z_{\bar{t}} \big)
= Q(A|X_s,X_u) Q(B|X_s,X_u),
\end{align*}
completing the proof.
\end{proof}
We are now ready to state and prove the last result of this paper, namely an analogue of Theorems~\ref{thm:sol_schro_additive} and~\ref{thm:discrete_schro} in the case of Brödinger. We do not know for the moment how to get an analogue of formula~\eqref{eq:sol_schro_additive} in this setting, but the case when $\TT$ is finite is tractable. 
\begin{theorem}
	\label{thm:bro}
Suppose that $R$ is reciprocal, and that problem~\eqref{eq:pb_bro} admits a (unique) solution $P$. Then $P$ is reciprocal.

In addition, if $R$ is irreducible in the sense of Assumption~\ref{ass:irreducibility}, and if $\TT := \{ t_1, \dots, t_K \}$ is finite, then there exist measurable functions $\eta, f_1, \dots, f_K$ with values in $[-\infty, + \infty)$ such that
 \begin{equation*}
 \frac{\D P}{\D R} = \exp\left( \eta(X_0, X_1) + \sum_{i=1}^K f_i(X_{t_i})\right).
 \end{equation*}
\end{theorem}
\begin{remark}[Dominated reciprocal measures]
	\label{rem:general_doesnt_work}
	By analogy with the Markov case and as already said in the introduction, we could expect that whenever $P$ is reciprocal and absolutely continuous with respect to $R$ reciprocal, then there exists a measurable function $\eta$ and a (regular) additive functional $A$ such that:
	\begin{equation*}
	\frac{\D P}{\D R} = \exp\Big( \eta(X_0, X_1) + A([0,1])\Big).
	\end{equation*}
	We do not know for the moment if this result is true or not, but at least Theorem~\ref{thm:bro} shows that this is true when $P$ is the solution of a discrete version of the Brödinger problem with respect to $R$.
	
	By calling $f := \exp(\eta)$, this formula writes:
	\begin{equation*}
	P =f(X_0, X_1) \times  \exp\Big( A([0,1])\Big) R.
	\end{equation*}
	In the case when $R$ is Markov, then $R':= \exp( A([0,1])) R$ is also Markov. Hence \emph{$P$ is a mixing of the bridges of the Markov measure $R'$ which is absolutely continuous with respect to $R$}. This means that the only way to build reciprocal measures $P$ from a Markov measure $R$ such that $P \ll R$ is to pick a new Markov reference measure  in the class of Markov measures which are dominated by $R$, and then to mix its bridges.
	
	We have a proof of this result in the  situation where $R$ is Markov and irreducible in the sense of Assumption~\ref{ass:markov_irreducibility}, under the important restriction that $R \ll P$. But we decided not to reproduce it here.
\end{remark}
\begin{proof}[Proof of Theorem  \ref{thm:bro}]
Let us take $\lambda \in (0,1)$, and let us call as before $P_\lambda := \Psi_\lambda {}\pf P$ and $R_\lambda := \Psi_\lambda {}\pf R$, $\Psi_\lambda$ being defined in~\eqref{eq:def_Psi_lambda}. As the relative entropy is invariant under push-forwards by injective functions, that is for all $Q\in\PO$, $H(\Psi_\lambda{}\pf Q| \Psi_\lambda {}\pf R) = H(P|R)$, problem~\eqref{eq:pb_bro} can be reformulated in terms of $P_\lambda$ and $R_\lambda$. For this, call:
\begin{equation*}
\TT_1 := \{ t \in [0,1] \mbox{ s.t. } \lambda t \in \TT \}\qquad \mbox{and}\qquad \TT_2 := \{ t \in [0,1] \mbox{ s.t. } 1 - (1-\lambda) t \in \TT \}.
\end{equation*}
Then, calling as before $(Z_t = (Z^1_t,Z^2_t))$ the canonical process on $C^0([0,1]; \XX \times \XX)$, $P_\lambda$ is easily seen as the solution of the following problem (which is  equivalent to~\eqref{eq:pb_bro}):
\begin{equation*}
H(Q|R_\lambda)\to \textrm{min};\quad Q\in\mathrm{P}(C^0([0,1]; \XX \times \XX)): \left\{ \begin{aligned}
Z^1_t{}\pf Q &= \mu_{\lambda t},\forall t\in\TT_1,\\
Z^2_t{}\pf Q &= \mu_{1 - (1-\lambda) t},\forall t\in\TT_2,\\
Q_0&=\pi.
\end{aligned}\right.
\end{equation*}
In particular, it is also the solution of the more constrained problem:
\begin{equation*}
H(Q|R_\lambda)\to \textrm{min};\quad Q\in\mathrm{P}(C^0([0,1]; \XX \times \XX)): \left\{ \begin{aligned}
Z_t{}\pf Q &= (P_\lambda)_t,\forall t\in\TT_1 \cup \TT_2,\\
Q_0&=\pi.
\end{aligned}\right.
\end{equation*}
But this one is exactly of Schrödinger type~\eqref{eq-11}. By Lemma~\ref{lem:reciprocal_markov}, $R_\lambda$ is Markov, so that by Theorem~\ref{thm:sol_schro_additive}, $P_\lambda$ is Markov, and as it is true for any $\lambda \in (0,1)$,  by Lemma~\ref{lem:reciprocal_markov} again, $P$ is reciprocal.

Now we suppose that $R$ is irreducible and that $\TT = \{t_1, \dots, t_K \}$ is finite, and we prove the second part of the statement. We suppose without loss of generality that $ 0 < t_1 < \dots < t_K$. For any $ 0<\lambda<1$ and all $t \in [0, \lambda]$, we call $\varphi(t):= 1 - (1-\lambda)t/\lambda$, so that with the same notations as before, $Z_{t/\lambda} = (X_t, X_{\varphi(t)})$. 
Let us choose $\lambda \in (t_K, 1)$ { so that all the $ \varphi(t_i)$'s are greater than $t_K$, implying $ \varphi(t_i)\neq t_j$ for any $i,j$}.

We have seen that $P_\lambda$ is the solution of a Schrödinger problem, constrained on the times $0, t_1 / \lambda, \dots, t_K / \lambda$. Moreover, we easily see that if $R$ is reciprocal and irreducible in the sense of Assumption~\ref{ass:irreducibility}, then $R_\lambda$ is not only Markov, but also irreducible in the sense of Assumption~\ref{ass:markov_irreducibility}. As a consequence, by Theorem~\ref{thm:discrete_schro}:
\begin{equation*}
\frac{\D P_\lambda}{\D R_\lambda} = \exp\Big( \eta(Z_0) + g_1(Z_{t_1 / \lambda}) + \dots + g_K(Z_{t_K/\lambda}) \Big).
\end{equation*}
for some measurable functions $\eta, g_1, \dots,g_K$.
By the fact that $\Psi_\lambda$ is injective, we easily deduce that:
\begin{equation*}
\frac{\D P}{\D R} = \exp\Big( \eta(X_0, X_1) + g_1(X_{t_1}, X_{\varphi(t_1)}) + \dots + g_K(X_{t_K}, X_{\varphi(t_K)}) \Big).
\end{equation*}

Hence, the only thing to prove is that each $g_i(X_{t_i}, X_{\varphi(t_i)})$ can be replaced by a function $f_i(X_{t_i})$. To do so, remark that by the same argument as in the proof of Theorem~\ref{thm:sol_schro_additive}, $\D P/\D R$ is $\sigma(X_0, X_{t_1}, \dots, X_{t_K}, X_1)$ measurable. As a consequence, there is a measurable function  $F(X_0, X_{t_1}, \dots, X_{t_K}, X_1)$ such that $R\ae$:
\begin{equation*}
F(X_0, X_{t_1}, \dots, X_{t_K}, X_1) =\eta(X_0, X_1) + g_1(X_{t_1}, X_{\varphi(t_1)}) + \dots + g_K(X_{t_K}, X_{\varphi(t_K)}).
\end{equation*}
Disintegrating  this expression with respect to $X_{\varphi(t_1)}, \dots, X_{\varphi(t_K)}$  and using Proposition~\ref{prop:conditioning}, we see that\\ for $R _{ \varphi(t_1),\dots, \varphi(t_K)}$-almost all $(y_1,\cdots,y_K),$ 
\begin{align}\label{eq-Fetag}
\begin{split}
F(X_0, X_{t_1}, \dots, X_{t_K}, X_1)=   \eta(X_0, X_1) + g_1(&X_{t_1}, y_1) + \dots + g_K(X_{t_K}, y_K),\\
 &R(  \sbt |X_{\varphi(t_1)}=y_1, \dots, X_{\varphi(t_K)}=y_K )\ae
\end{split}
\end{align}
On the other hand, we know with Lemma~\ref{lem:tensorization} that
\begin{align*}
R_{0, t_1, \dots, t_K, 1} \ll R_0\otimes_{i=1}^K R_{t_i} \otimes R_1.
\end{align*}
and
\begin{equation*}
  R_0\otimes_{i=1}^K R_{t_i} \otimes_{i=1}^K R_{\varphi(t_{K-i})}\otimes R_1
  	\ll R_{0, t_1, \dots, t_K, \varphi(t_K), \dots, \varphi(t_1), 1}.
\end{equation*}
Disintegrating this last relation with respect to $X_{\varphi(t_1)}, \dots, X_{\varphi(t_K)}$ thanks to Lemma \ref{lem:CT} gives us
\begin{equation*}
 R_0\otimes_{i=1}^K R_{t_i} \otimes R_1 \ll R((X_0, X_{t_1}, \dots, X_{t_K}, X_1)\in \sbt |X_{\varphi(t_1)}=y_1, \dots, X_{\varphi(t_K)}=y_K ),
\end{equation*}
and taking the first relation into account, we arrive at
\begin{align*}
 R_{0, t_1, \dots, t_K, 1}  \ll R((X_0, X_{t_1}, \dots, X_{t_K}, X_1)\in \sbt |X_{\varphi(t_1)} = y_1, \dots, X_{\varphi(t_K)}= y_K ),
\end{align*}
for $R _{ \varphi(t_1),\dots, \varphi(y_K)}$-almost all $(y_1,\cdots,y_K).$
\\
With \eqref{eq-Fetag}, this implies that 
one  can choose $y_1, \dots, y_K$ (in a set with full mass with respect to $R _{ \varphi(t_1),\dots, \varphi(t_K)}$) such that 
\begin{align*}
F(X_0, X_{t_1}, \dots, X_{t_K}, X_1)=   \eta(X_0, X_1) + g_1(X_{t_1}, y_1) +& \dots + g_K(X_{t_K}, y_K),
\quad R\ae
\end{align*}
 The result follows  by choosing $f_i(X_{t_i}) := g_i(X_{t_i}, y_i)$.
\end{proof}

\appendix
\section{Regularity of two indices martingales}
\label{regularitymartingale}
In \cite{bakry1979regularite}, Bakry generalizes the classical \textit{\cadlag} regularity results for martingales in the case two indices martingales. Let us present briefly his results and explain how it is used in our context.

Take $(\Omega, \mathcal{G}, \PP)$ a complete probability space and $(\mathcal{G}^1_u)_{u \in \RR_+}$ and $(\mathcal{G}^2_v)_{v \in \RR_+}$ two right continuous and complete filtrations. We introduce the two indices filtration $(\mathcal{G}_{u,v})_{(u,v) \in \RR_+ \times \RR_+}$ defined for all $(u,v) \in \RR_+ \times \RR_+$ by $\mathcal{G}_{u,v} = \mathcal{G}^1_{u} \cap \mathcal{G}^2_{v}$. We assume the following independence condition:
\begin{quote}
(H) for all $(u,v) \in \RR_+ \times \RR_+$, $\mathcal{G}^1_{u}$ and $\mathcal{G}^2_{v}$ are independent conditionally on $\mathcal{G}_{u,v}$.
\end{quote}

On $\RR_+ \times \RR_+$ we define the binary relations
\begin{gather*}
\mbox{forall } (u,v)\mbox{ and } (\mu, \nu) \in \RR_+ \times \RR_+, \quad (u,v) \preceq (\mu, \nu) \, \Leftrightarrow\, u \leq \mu \mbox{ and } v \leq \nu,\\
\mbox{forall } (u,v)\mbox{ and } (\mu, \nu) \in \RR_+ \times \RR_+, \quad (u,v) \prec (\mu ,\nu) \, \Leftrightarrow\, u < \mu \mbox{ and } v < \nu.
\end{gather*}
The first one is a partial order and the second one is a strict partial order.

In this setting, a martingale is a process $M = (M_{u,v})$ such that fixing $u$ (resp. $v$), $M$ is a $(\mathcal{G}^1_u)$-martingale (resp. $(\mathcal{G}^2_v)$-martingale). This is equivalent with the fact that
\begin{itemize}
\item for all $(u,v) \in \RR_+ \times \RR_+$, $M_{u,v}$ belongs to $L^1(\Omega)$ and is $\mathcal{G}_{u,v}$-measurable,
\item for all $(u,v) \preceq (\mu, \nu)$, 
\[
\EE [M_{\mu, \nu} | \mathcal{G}_{u,v}] = M_{u,v} \quad \mbox{a.e.}
\]
\end{itemize}

Fixing $\omega$, a trajectory is said to be right continuous if for all $(u,v) \succeq (0,0)$,
\[
\lim_{\substack{(\mu, \nu) \to (u,v) \\ (\mu, \nu) \succeq (u,v) }} M_{\mu, \nu} = M_{u,v}.
\] 
It is said to be left limited if for all $(u,v) \succ (0,0)$,
\[
\lim_{\substack{(\mu, \nu) \to (u,v) \\ (\mu, \nu) \prec (u,v) }} M_{\mu, \nu} \mbox{ exists}.
\] 
Remark that in the first case, we use $\preceq$ while in the second one, that is $\prec$ that comes into play.

The main result we will use is the following.
\begin{theorem}
Let $M = (M_{u,v})$ a two indices martingale with the following additional property:
\[
\forall (u,v) \in \RR_+ \times \RR_+, \quad \EE\Big[|M_{u,v}| \log^+\big(|M_{u,v}|\big)\Big] < \infty.
\]
Then $M$ admits a modification such that for every $\omega$, $M(\omega)$ is right continuous and left limited. In particular, $M_{u,v}$ is well defined for all $\omega$ for all $(u,v)$.
\end{theorem}
Let us explain now how to use this result in our setting. In the case of Theorem~\ref{res-02}, we want to show some regularity for a closed martingale of the following type:
 \[
 D_{s,t} := \EE_R[D | X_{[s,t]}]
 \]
with
\begin{itemize}
\item $D \geq 0$ and $\EE[D \log D] < + \infty$,
\item $(\sigma(X_{[0,t]})_{t \in [0,1]}$ and $(\sigma(X_{[s,1]})_{s \in [0,1]}$ are respectively right and left continuous,
\item $R$ is Markov.
\end{itemize}
We define
\begin{itemize}
\item for all $u \in \RR_+$ $\mathcal{G}^1_u := \sigma(X_{[1-u\wedge 1, 1]})$ (which is right continuous),
\item for all $v \in \RR_+$, $\mathcal{G}^2_v := \sigma(X_{[t \wedge 1, 1]})$ (which is right continuous),
\item $M = (M_{u,v})_{(u,v) \in \RR_+ \times \RR_+}$ is defined from $(D_{s,t})$ by the change of variable $(s,t) = (1-u, v)$ as in Figure~\ref{defM}.
\end{itemize}

\begin{figure}
\centering
\begin{tikzpicture}[scale=1]
\fill[green!20] (0,0) -- (0,2) -- (2,0) -- cycle;
\fill[red!20] (0,2) -- (2,2) -- (2,0) -- cycle;
\fill[yellow!20] (0,2) -- (2,2) -- (2,4) -- (0,4) -- cycle;
\fill[yellow!20] (2,0) -- (2,2) -- (4,2) -- (4,0) -- cycle;
\fill[blue!20] (2,2) rectangle (4,4);
\draw[thick, ->] (-0.3,0) -- (4,0) node[below]{$u$};
\draw[thick, ->] (0,-0.3) -- (0,4) node[left]{$v$};
\draw (-5pt,-7pt) node{$0$};
\draw (2, -2pt) -- (2, 2pt) node[below=5pt]{$1$};
\draw (-2pt, 2) -- (2pt, 2) node[left=5pt]{$1$};
\draw[thick, red] (0, 2) -- (2, 0);
\draw (0.6,0.5) node{$\EE[D]$};
\draw (1.4, 1.4) node{$D_{1-u, v}$};
\draw (3,1) node{$D_{0,v}$};
\draw (1,3) node{$D_{1-u,1}$};
\draw (3,3) node{$D_{0,1}$};
\end{tikzpicture}
\caption{ Definition of $M$ from $D$.} \label{defM}
\end{figure}
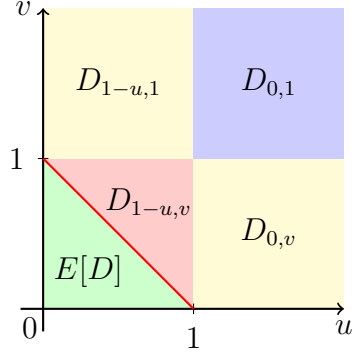
The assumption on the $L \log L$ integrability of $M$ is trivial with Jensen inequality, so we only need to check assumption (H). It is quite clear that it is equivalent to the fact that for all $s \leq t$ in $[0,1]$, $\sigma(X_{[s,1]})$ and $\sigma(X_{[0,t]})$ are independent conditionally on $\sigma(X_{[s,t]})$. But this is a direct application of the Markov property of $R$. As a consequence, $M$ has a modification which is right continuous and left limited. In particular, as far as $D$ is concerned and as illustrated in Figure~\ref{regularityD}, we get up to a modification:
\begin{gather*}
\forall \, 0 \leq s \leq t \leq 1, \, (s,t) \neq (0,1), \quad \lim_{\substack{(\sigma, \tau) \to (s,t) \\
\sigma \leq s, \, \tau \geq t}} D_{\sigma, \tau} = D_{s,t},\\
\forall 0 \leq s < t \leq 1, \quad \lim_{\substack{(\sigma, \tau) \to (s,t) \\
\sigma > s, \, \tau < t}} D_{\sigma, \tau} \mbox{ exists}.
\end{gather*}
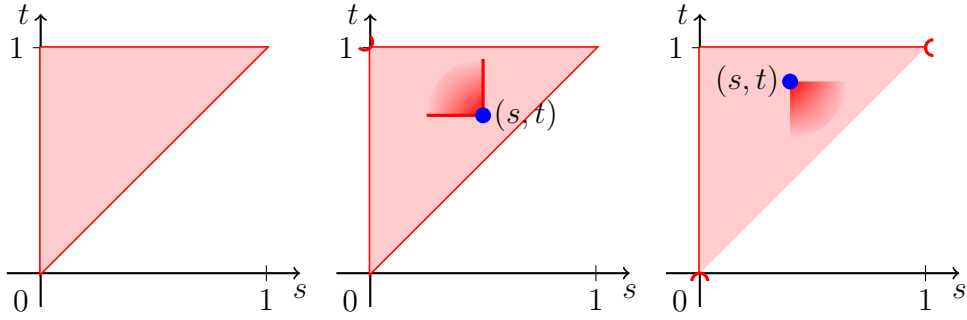
\begin{figure}
\begin{tikzpicture}[scale=1.5]
\draw[thick, ->] (-0.3,0) -- (2.3,0) node[below]{$s$};
\draw[thick, ->] (0,-0.3) -- (0,2.3) node[left]{$t$};
\draw (-5pt,-7pt) node{$0$};
\draw (2, -2pt) -- (2, 2pt) node[below=5pt]{$1$};
\draw (-2pt, 2) -- (2pt, 2) node[left=5pt]{$1$};
\draw[very thick, red] (0,0) -- (2,2) -- (0,2) -- cycle;
\fill[red!20] (0,0) -- (2,2) -- (0,2) -- cycle;
\end{tikzpicture}
\begin{tikzpicture}[scale=1.5]
\draw[thick, ->] (-0.3,0) -- (2.3,0) node[below]{$s$};
\draw[thick, ->] (0,-0.3) -- (0,2.3) node[left]{$t$};
\draw (-5pt,-7pt) node{$0$};
\draw (2, -2pt) -- (2, 2pt) node[below=5pt]{$1$};
\draw (-2pt, 2) -- (2pt, 2) node[left=5pt]{$1$};
\draw[very thick, red] (0,0) -- (2,2) -- (0,2) -- cycle;
\fill[red!20] (0,0) -- (2,2) -- (0,2) -- cycle;
\shade[lower right=red, upper left=red!20, upper right=red!20, lower left=red!20] (1,1.4) -- (1,1.9) arc (90: 180: 0.5) -- cycle;
\draw (1,1.4) node[right]{$(s,t)$};
\draw[very thick, red] (1,1.4) -- (1,1.9);
\draw[very thick, red] (1,1.4) -- (0.5,1.4);
\fill[blue] (1,1.4) circle (2pt);
\draw[very thick, red] (0,2) arc (-45:45:2pt);
\draw[very thick, red] (0,2) arc (-45:-135:2pt);
\end{tikzpicture}
\begin{tikzpicture}[scale=1.5]
\draw[thick, ->] (-0.3,0) -- (2.3,0) node[below]{$s$};
\draw[thick, ->] (0,-0.3) -- (0,2.3) node[left]{$t$};
\draw (-5pt,-7pt) node{$0$};
\draw (2, -2pt) -- (2, 2pt) node[below=5pt]{$1$};
\draw (-2pt, 2) -- (2pt, 2) node[left=5pt]{$1$};
\draw[very thick, red] (2,2) -- (0,2);
\draw[very thick, red] (0,0) -- (0,2);
\fill[red!20] (0,0) -- (2,2) -- (0,2) -- cycle;
\shade[lower right=red!20, upper left=red, upper right=red!20, lower left=red!20] (0.8,1.7) -- (1.3,1.7) arc (0: -90: 0.5) -- cycle;
\fill[blue] (0.8,1.7) circle (2pt);
\draw (0.8,1.7) node[left]{$(s,t)$};
\draw[very thick, red] (2,2) arc (-180:-90:2pt);
\draw[very thick, red] (2,2) arc (180:90:2pt);
\draw[very thick, red] (0,0) arc (90:0:2pt);
\draw[very thick, red] (0,0) arc (90:180:2pt);
\end{tikzpicture}
\caption{\label{regularityD} To the left, the set of definition of $D$. In the middle, the continuity property of $D$. To the right, the "left limit" property of $D$.}
\end{figure}
This is the analogue of the regularity condition for contents stated in Definition~\ref{def:regular_content}.

\section{Conditioning trick}
\label{app:CT}
Let us review a few elementary properties of  conditioning  and absolute continuity.
Let $\mathcal{A}$ and $ \mathcal{B}$ be two Polish spaces  equipped with the their Borel $ \sigma$-fields, and let  $\qq$ and $\pp$ be respectively a nonnegative $ \sigma$-finite measure and a probability measure on   $\mathcal{A}$. 
It is assumed that $\mathcal{A}$ and $ \mathcal{B}$ are Polish to ensure, for any measurable mapping  $\phi:\mathcal{A}\to \mathcal{B},$  the disintegration formula
\begin{align*}
\qq=\int _{ \mathcal{B}}\qq ^{ \phi=b}\, \qq_\phi(\D b),
\end{align*}
where $\qq_\phi:= \phi\pf\qq$ and $\qq ^{ \phi=b}:=\qq(\sbt\mid\phi=b)$ is uniquely well defined $\forall b,\ \qq_\phi\ae$ We refer to \cite{dellacherie1978probabilities}[III, 70] for this standard result.

\begin{proposition}
	\label{prop:conditioning}
For any measurable mapping  $\phi:\mathcal{A}\to \mathcal{B},$ we have
\begin{align*}
\pp\ll\qq
\iff
\left\{
\begin{array}{ll}
\pp_\phi\ll\qq_\phi,& \\
\pp ^{ \phi=b} \ll\qq ^{ \phi=b},&\quad  \forall b,\ \pp_\phi\ae
\end{array}\right.
\end{align*}
\end{proposition}

\begin{proof}
To show that $\pp\ll\qq$ implies $\pp_\phi\ll\qq_\phi,$ it is enough to remark that $0=\qq_\phi(B):=\qq(\phi ^{ -1}(B)),$ implies that $\pp_\phi(B):=\pp(\phi ^{ -1}(B))=0.$

Let us show that $\forall b,\ \pp_\phi\ae$, $\pp ^{ \phi=b} \ll\qq ^{ \phi=b}$. For this, let us call $f := \D \pp / \D \qq$ and $g := \D\ (\phi \pf \pp) / \D\ (\phi \pf \qq)$. It is straightforward to check that
\begin{equation*}
\tilde\pp^b := \frac{f}{g(b)} \qq ^{ \phi=b}
\end{equation*}
is well defined $\forall b,\ \pp_\phi\ae$ and that it solves the problem of disintegrating $\pp$ with respect to $\phi$. Hence, by uniqueness of the disintegration, $\forall b,\ \pp_\phi\ae$, $\pp ^{ \phi=b} = \tilde\pp^b$. The result follows.
\\
The converse part of the statement is an easy consequence of the disintegration formula.
\end{proof}

Let $\XX$ and $\YY$ be two Polish spaces equipped with their Borel $ \sigma$-fields. Denoting $(X,Y)$  the identity on $\XX\times\YY,$ $\qq_X\in \mathrm{M}(\XX), \qq_Y\in \mathrm{M}(\YY)$ are the marginal measures of $\qq$.

\begin{lemma}[Conditioning trick] \label{lem:CT}
Assume that $\qq\in \mathrm{M}(\XX\times \YY)$ satisfies  \[\qq_X\otimes\qq_Y\ll\qq \ll \qq_X\otimes\qq_Y .\]
Then, for $\qq_Y$-almost all $y\in\YY,$ we have $\qq_X\ll \qq ^{ Y=y}_X:= \qq ^{ Y=y}(X\in\sbt)\ll \qq_X$.\\
 In particular, if some property holds $\qq_X ^{ Y=y}\ae,$ for any $y$ in  some $\YY_o$ with $\qq_Y(\YY_o)>0,$ then this property holds $\qq_X\ae$
\end{lemma}

\begin{proof}
Applying previous proposition with $ \mathcal{A}=\XX\times\YY,$ $\pp=\qq_X\otimes\qq_Y$ and  $\phi=Y$, we obtain $\qq_X\ll \qq_X ^{ Y=y},\ \forall y, \qq_Y\ae$ We obtain the other inequality by reversing the roles of $\qq$ and $\qq_X\otimes\qq_Y$. Last statement is a consequence of the definition of the absolute continuity, because:
$
\pp\ll\rr
\implies
[\rr\ae \textrm{ implies }\pp\ae].
$
\end{proof}

\bibliographystyle{plain}
\bibliography{BL1.bib}

\begin{thebibliography}{10}

\bibitem{arn17}
M.~Arnaudon, A.B. Cruzeiro, C.~L{\'e}onard, and J.-C. Zambrini.
\newblock An entropic interpolation problem for incompressible viscid fluids.
\newblock To appear in Ann. Inst. H. Poincar\'e Probab. Statist..
  arXiv:1704.02126.

\bibitem{bakry1979regularite}
D~Bakry.
\newblock Sur la r{\'e}gularit{\'e} des trajectoires des martingales {\`a} deux
  indices.
\newblock {\em Probability Theory and Related Fields}, 50(2):149--157, 1979.

\bibitem{baradat2018existence}
A.~Baradat.
\newblock On the existence of a scalar pressure field in the br\"odinger
  problem.
\newblock To appear in SIAM Journal on Mathematical Analysis.

\bibitem{baradat2019small}
A.~Baradat and L.~Monsaingeon.
\newblock Small noise limit and convexity for generalized incompressible flows,
  {S}chr{\"o}dinger problems, and optimal transport.
\newblock To appear in Archive for Rational Mechanics and Analysis.

\bibitem{benamou2015iterative}
J-D. Benamou, G.~Carlier, M.~Cuturi, L.~Nenna, and G.~Peyr{\'e}.
\newblock Iterative {B}regman projections for regularized transportation
  problems.
\newblock {\em SIAM J. Sci. Comput.}, 37(2):A1111--A1138, 2015.

\bibitem{Beu60}
A.~Beurling.
\newblock An automorphism of product measures.
\newblock {\em Ann. of Math.}, 72:189--200, 1960.

\bibitem{BLN}
J.M. Borwein, A.S. Lewis, and R.D. Nussbaum.
\newblock Entropy minimization, {DAD} problems and doubly stochastic kernels.
\newblock {\em J. Funct. Anal.}, 123:264--307, 1994.

\bibitem{bre89}
Y.~Brenier.
\newblock The least action principle and the related concept of generalized
  flows for incompressible perfect fluids.
\newblock {\em J. Amer. Math. Soc.}, 2(2):225--255, 1989.

\bibitem{Csi75}
I.~Csisz\'ar.
\newblock {$I$}-divergence geometry of probability distributions and
  minimization problems.
\newblock {\em Annals of Probability}, 3:146--158, 1975.

\bibitem{cuturi2013sinkhorn}
M.~Cuturi.
\newblock Sinkhorn distances: Lightspeed computation of optimal transport.
\newblock In {\em Advances in neural information processing systems}, pages
  2292--2300. NIPS 2013, 2013.
\newblock arXiv:1306.0895.

\bibitem{dacunha2012probability}
D.~Dacunha-Castelle and M.~Duflo.
\newblock {\em Probability and statistics, Vol. 2}.
\newblock Springer Science \& Business Media. Springer Verlag, 2012.

\bibitem{daw87}
D.A. Dawson and J.~G{\"a}rtner.
\newblock Large deviations from the {M}c{K}ean-{V}lasov limit for weakly
  interacting diffusions.
\newblock {\em Stochastics}, 20:247--308, 1987.

\bibitem{dellacherie1978probabilities}
Claude Dellacherie and Paul-Andr{\'e} Meyer.
\newblock {Probabilities and potential, vol. 29 of North-Holland Mathematics
  Studies}, 1978.

\bibitem{follmer1988random}
H.~F{\"o}llmer.
\newblock {\em Random fields and diffusion processes, in \'Ecole d'\'et{\'e} de
  Probabilit{\'e}s de Saint-Flour XV-XVII-1985-87}, volume 1362 of {\em Lecture
  Notes in Mathematics}.
\newblock Springer, Berlin, 1988.

\bibitem{halmos1950measure}
P.~Halmos.
\newblock {\em Measure Theory}.
\newblock Number~18 in Graduate text in mathematics. Springer, 1950.

\bibitem{Leo01b}
C.~L{\'e}onard.
\newblock Minimizers of energy functionals.
\newblock {\em Acta Math. Hungar.}, 93(4):281--325, 2001.

\bibitem{leonard2014some}
C.~L{\'e}onard.
\newblock Some properties of path measures.
\newblock In {\em S\'eminaire de probabilit\'es de Strasbourg, vol. 46.}, pages
  207--230. Lecture Notes in Mathematics 2123. Springer., 2014.

\bibitem{leo13}
C.~L{\'e}onard.
\newblock A survey of the {S}chr{\"o}dinger problem and some of its connections
  with optimal transport.
\newblock {\em Discrete Contin. Dyn. Syst. A}, 34(4):1533--1574, 2014.

\bibitem{leonard2014reciprocal}
C.~L\'{e}onard, S.~R{\oe}lly, and J-C. Zambrini.
\newblock Reciprocal processes. {A} measure-theoretical point of view.
\newblock {\em Probab. Surv.}, 11:237--269, 2014.

\bibitem{leonard2012schrodinger}
Christian L{\'e}onard.
\newblock {From the Schr{\"o}dinger problem to the Monge--Kantorovich problem}.
\newblock {\em Journal of Functional Analysis}, 262(4):1879--1920, 2012.

\bibitem{Mika04}
T.~Mikami.
\newblock Monge's problem with a quadratic cost by the zero-noise limit of
  $h$-path processes.
\newblock {\em Probab. Theory Relat. Fields}, 129:245--260, 2004.

\bibitem{RT93}
L.~R{\"u}schendorf and W.~Thomsen.
\newblock Note on the {S}chr{\"o}dinger equation and {$I$}-projections.
\newblock {\em Statist. Probab. Lett.}, 17:369--375, 1993.

\bibitem{RT98}
L.~R{\"u}schendorf and W.~Thomsen.
\newblock Closedness of sum spaces and the generalized {S}chr\"odinger problem.
\newblock {\em Theory Probab. Appl.}, 42(3):483--494, 1998.

\bibitem{sch31}
E.~Schr\"odinger.
\newblock {\"U}ber die {U}mkehrung der {N}aturgesetze.
\newblock {\em Sitzungsberichte Preuss. Akad. Wiss. Berlin. Phys. Math.},
  144:144--153, 1931.

\bibitem{sch32}
E.~Schr{\"o}dinger.
\newblock Sur la th{\'e}orie relativiste de l'{\'e}lectron et
  l'interpr{\'e}tation de la m{\'e}canique quantique.
\newblock {\em Ann. Inst. H. Poincar\'e}, 2:269--310, 1932.
\newblock 

\bibitem{sinkhorn1964relationship}
Richard Sinkhorn.
\newblock A relationship between arbitrary positive matrices and doubly
  stochastic matrices.
\newblock {\em The annals of mathematical statistics}, 35(2):876--879, 1964.

\bibitem{sinkhorn1967diagonal}
Richard Sinkhorn.
\newblock Diagonal equivalence to matrices with prescribed row and column sums.
\newblock {\em The American Mathematical Monthly}, 74(4):402--405, 1967.

\bibitem{wagner1977survey}
Daniel~H Wagner.
\newblock Survey of measurable selection theorems.
\newblock {\em SIAM Journal on Control and Optimization}, 15(5):859--903, 1977.

\end{thebibliography}

\end{document}